\documentclass[a4paper]{article}

\usepackage[utf8]{inputenc}             
\usepackage[T1]{fontenc}                
\usepackage[english]{babel}              
\usepackage{graphicx}                   
\usepackage{amsmath,amsfonts,amssymb}   
\usepackage{minted}                   
\usepackage[all]{xy}
\usepackage{amsthm}
\usepackage{bm}
\usepackage{geometry}
\usepackage{algorithmic}
\usepackage{hyperref}
\usepackage{appendix}
\usepackage{subcaption}
\usepackage{color}
\usepackage{enumitem}
\usepackage{authblk}
\usepackage{lmodern} 
\usepackage{bm}
\usepackage{mathtools}
\usepackage{nccmath}
\usepackage{ stmaryrd }
\usepackage{soul}
\usepackage[dvipsnames]{xcolor}
\usepackage[linesnumbered,ruled,vlined]{algorithm2e}
\usepackage[nameinlink, capitalise]{cleveref}
\usepackage[
    backend=biber,
    uniquename=false,
    bibencoding=utf8,
    sorting=nyt,
    doi=false, isbn=false, url=false,
    style=alphabetic,
    maxcitenames=2,
    uniquelist=false,
    maxbibnames=123,
    backref=false,  
    hyperref=true
]{biblatex}
\hypersetup{ colorlinks=true, 
    linktoc=all,     
    linkcolor=blue,  
} \PassOptionsToPackage{unicode}{hyperref}
\PassOptionsToPackage{naturalnames}{hyperref}

\DeclareMathOperator{\R}{\mathbb{R}}

\DeclareMathOperator{\id}{id}

\newcommand{\Diff}{\operatorname{Diff}}
\newcommand{\Isom}{\text{Isom}}

\newcommand{\app}[4]{\begin{array}{ccl}
   #1 & \longrightarrow & #2 \\
   #3 & \longmapsto & #4 \\
\end{array}}

\newcommand{\apps}[4]{\begin{array}{ccl}
   #1 & \to & #2 \\
   #3 & \mapsto & #4 \\
\end{array}}

\def\bp{{\bm{p}}}
\def\bq{{\bm{q}}}

\def\bqS{{\bm{q_S}}}

\newcommand{\D}[1]{\mathsf{D}_{#1}}

\newtheorem{defi}{Definition}[section]
\newtheorem{prop}[defi]{Proposition}
\newtheorem{theo}[defi]{Theorem}
\newtheorem{lemma}[defi]{Lemma}

\newtheorem{rem}[defi]{Remark}
\newtheorem{exam}[defi]{Example}
\addto\captionsfrench{}
\title{Decoupling actions of  finite-dimensional Lie groups and of groups of diffeomorphisms in the large deformation framework}

\author{
  Rayane Mouhli\thanks{Université Paris Cité, MAP5 — \texttt{rayane.mouhli@math.cnrs.fr}} \quad
  Thomas Pierron\thanks{ENS Paris-Saclay, Centre Borelli — \texttt{thomas.pierron@ens-paris-saclay.fr}}
}
\date{} 

\begin{document}
\maketitle

\begin{abstract}
In computational anatomy, the Large Deformation Diffeomorphic Metric Mapping (LDDMM) framework has become a central tool for modeling smooth, invertible transformations between shapes such as curves or landmarks.
In this paper, we extend this framework  by enriching diffeomorphic deformations with transformations induced by finite-dimensional Lie groups (e.g.~isometries, scalings), and we develop a registration model that decouples the actions of these two types of deformation on the shape during the matching process. To achieve this, we consider semidirect products between finite-dimensional groups and groups of diffeomorphisms, endowed with a right-invariant sub-Riemannian structure that give rise to new variational problems for shape registration. By exploiting symmetries and reduction theory, we decouple the contributions of each group throughout the matching process. We further extend the framework to incoroporate anisotropic deformations that preferentially favor certain directions during registration. On the numerical side, we propose an algorithm based on a joint optimization over both deformation groups, in contrast to the standard two-stage approach that optimizes first over the finite-dimensional component and then over the diffeomorphic one. Experiments on curves and landmarks demonstrate that the proposed joint optimization improves registration accuracy and more effectively disentangles the contributions of the two deformation groups.

\end{abstract}
 
\tableofcontents
\newpage

\section{Introduction}

Shape analysis is a field of study recently developed in response to the increase of innovations in medical imaging techniques. Its purpose is to compare several shapes among a family of shapes to study their variability and perform statistics on them. In particular, the deformation of a source shape onto a target shape is a central topic since it enables to track the difference of geometric properties between shapes. A well-known method for such registration is the Large Deformation Diffeomorphic Metric Mapping (LDDMM) framework introduced by Beg et al. \cite{BegMillerTrouveYounes2005}, where shapes are deformed by flows of diffeomorphisms defined on an ambient space. These diffeomorphisms are obtained by integrating time-varying vector fields belonging to a Reproducing Kernel Hilbert Space $V$(RKHS) \cite{Aro50,glaunes2005transport} from which the energy of the deformation is defined. In this setting, given a source $q_S$ and a target $q_T$ shape (curves, landmarks, images etc.), the registration task can be expressed as the following minimization problem
\begin{eqnarray}
    \label{var_pb_LDDMM}
    \inf_{v \in L^2([0,1],V)} J(v) &=& \int_0^1 \frac{1}{2}  \vert v_t \vert_V^2 \, dt + \mathcal{D}(\varphi_1) \\
     \text{ s.t }& &     \left\{
        \begin{array}{l}
            \dot{\varphi}_t =  v_t \circ \varphi_t \\
            \varphi_0 = \id        
        \end{array} 
        \right.\notag
\end{eqnarray}
where $\mathcal{D}$ is a data attachment term measuring the distance between the target and the deformed shape at final time.

For a wide range of applications, the deformation aligning a source with a target can be decomposed into an affine motion (e.g isometries, translations, scalings, etc.) and a diffeomorphic one. Thus, for statistical analysis of shape transformations, it is important to distinguish the contribution of each type of deformations, which is not possible in a classic large deformation model. In practice, a first step is often performed to affinely align the source and the target, which implies that the affine registration is completely independent of the diffeomorphic one. Such a two-step procedure can produce incoherent results due to this independence, which motivates the development of models that handle both deformations simultaneously. 

A first approach for combining several types of deformations has been introduced by Bruveris et al. \cite{MMRI}, based on a semidirect product of groups encoding both coarse and fine deformation scales. In the particular case of a semidirect product of groups of diffeomorphisms \cite{BruRiVia}, this model is actually equivalent to considering a single type of deformation model generated by RKHS induced by a sum of kernels \cite{risser2010,risser2011simultaneous}. Gris et al. \cite{gris2015sub} defined the modular deformation framework allowing the combination of several structured vector fields to generate more complex deformations. More recently, a multiscale framework has been developed \cite{gga}, following the work of \cite{BME}, to perform a sequential matching, from the coarser to the finer scale, with a group of deformations as a direct product of a finite dimensional Lie group representing affine motion and group of diffeomorphisms. Other multiscale approaches are presented in \cite{debroux2023multiscale,MODIN20191009,sommer2011,sommer2013sparse}.

In this paper, we extend the LDDMM setting by adapting the framework developed in \cite{general_setting} and further elaborated in \cite{gga, Pierron2024}, to consider groups of deformations constructed as semidirect products of finite-dimensional Lie groups and groups of diffeomorphisms. Thus, we consider a general shape space $\mathcal{Q}$, modeled as a Banach manifold, on which this group of deformations acts. In particular, we study the differential structure of this group, the differentiability of its action, as well as the definition of right-invariant sub-Riemannian structures on it, to formulate new matching variational problems. This general theoretical framework encompasses the particular example of rigid motions and scalings. An additional challenge is to ensure the decoupling of the actions of both groups on a shape during the registration process. Indeed, if the finite-dimensional group is the group of rotations, any statistical study aimed at quantify the rotation of the shape becomes biased when the diffeomorphism contributes to the rotation of the shape. Therefore, decoupling the two types of deformation is essential for a meaningful analysis of the total deformation.  We achieve this decoupling thanks to a reduction method \cite{MARSDEN1974121,marsden1994introduction,Satzer1977} that allows to quotient out the action of the finite dimensional group on the total deformation. 
Indeed, this reduction method leads to a more coherent parameterization of both types of deformation, in the sense that the statistical analysis performed on each component becomes more accurate and more representative of the underlying geometric deformations.

\paragraph{Main contributions.} We present the main contributions of this article :

\begingroup
\renewcommand\labelenumi{(\theenumi)}
\begin{enumerate}
    \item Considering a finite-dimensional Lie group $G$ and $\Diff_{C_0^k}(\R^d)$ the group of $C^k$-diffeomorphisms that tends to identity, we define the sub-Riemannian structure  on the group of deformations $G \ltimes \Diff_{C_0^k}(\R^d)$ and its action on a shape space $\mathcal{Q}$ in the context of a matching problem. To add informations on the shape and ease the matching, we consider the augmented shape space $\tilde{\mathcal{Q}}=G \times \mathcal{Q}$. The addition of the finite dimensional Lie group to the shape space allows to get more informations about the shape, as orientation for rotations or position for translations. In particular, we introduce a change of variable $\tilde{q}=g^{-1}\cdot q$ that removes the contribution of $G$ on the shape. It also allows to consider the diffeomorphic deformation  in a natural reference frame associated to the shape, without influence of the finite dimensional group. 

    \item For a space of vector fields $V \hookrightarrow C_0^{k+2}(\R^d,\R^d)$ that is $G$-invariant, we present a reduction method that allows to decouple the action of $G$ and $\Diff_{C_0^k}(\R^d)$ on $\mathcal{Q}$ by satisfying constraints on the momentum $\mu = 0$  where $\mu : T^*G \oplus T^*\mathcal{Q} \to \mathfrak{g}^*$. For an Hamiltonian $H : T^*G \oplus T^* \mathcal{Q} \to \R$, Hamiltonian equations characterize geodesics of the matching problem. By restricting the Hamiltonian to the zero level set $\mu^{-1}(0)$, the Hamiltonian flow on $T^*G \oplus \mu^{-1}(0)$  can be projected on $T^*G \oplus T^*(\mathcal{Q}/G)$ which can be interpreted as a decoupling between the deformations induced by $G$ and $\Diff_{C_0^k}(\R^d)$.

    \item To favorize certain directions in the deformations, instead of considering a space of vector fields $V$ that is a RKHS induced by a scalar Gaussian kernel, which is the classic framework, we introduce an anisotropic Gaussian kernel associated with an anisotropic metric $\Sigma \in S_d^{++}$, defined by $k_{\Sigma}(x,y)=\exp(-\frac{1}{2}\Vert x-y\Vert_{\Sigma^{-1}})\Sigma$ where $\langle x,y\rangle_{\Sigma}=\langle \Sigma x,y\rangle$. Then, we consider diffeomorphic deformations induced by the space of vector fields $V_{\Sigma}$ which is a RKHS with kernel $k_{\Sigma}$. To keep track of the favorized axis given by the metric $\Sigma$ during the deformation, we enrich the shape space $\mathcal{Q}$ by adding the metric to define a new shape space $S_d^{++} \times \mathcal{Q}$, which allows to transport those axis along with the shape under the action of scalings, isometries and diffeomorphisms  $ \alpha-\operatorname{Isom}(\R^d) \ltimes \Diff_{C_k^0}(\R^d)$ where $\alpha-\operatorname{Isom}(\R^d):=(\R_{>0} \times SO_d)\ltimes \R^d$.

\end{enumerate}

\endgroup

\paragraph{Structure of the paper.} Section \ref{Sec:general_framework} presents the general framework of the paper by the introduction of the group of deformations $G \ltimes \Diff_{C_0^k}(\R^d)$ and its sub-Riemannian structure. It also presents the inexact matching problem on a shape space and a reduction method to decouple the actions of both groups of deformations. Section \ref{Sec:Rigid+diffeo} focuses on the applications of the framework from the previous section to the case of curves that are deformed by isometries $G = \Isom(\R^d)$ and diffeomorphisms. Finally, section \ref{Ani:Sec:sp_const_kernels} introduces anisotropic deformation, represented by an anisotropic metric $\Sigma \in S_d^{++}$ that can be transported along the shape during the global deformation to keep the anisotropies in the same reference frame as the shape. 

\section{General framework}\label{Sec:general_framework}
In this section, we introduce the general theoretical framework for coupling the classic large deformation model (LDDMM)\cite{BegMillerTrouveYounes2005} with the action of a finite-dimensional Lie group. Such a construction will serve as foundation for the applications presented in the subsequents sections. The differential structure and the regularity properties of the group of deformations considered in this paper is presented in section \ref{Sec:diff_structure}. We introduce a right-invariant sub-Riemannian metric on this group and recall the classical completeness results. Following \cite{general_setting, gga}, we then specify in section \ref{Sec:shape_space} the hypotheses required on the differents actions and spaces to induce a sub-Riemannian structure on the space of shapes. This will serve to define a distance between shapes that will be used as a metric to introduce a variational problem for the matching of shapes, taking into account the deformations of the group. 
\subsection{A semidirect product}
\label{Sec:diff_structure}
In classical large deformation model \cite{Trouve1998,BegMillerTrouveYounes2005,general_setting}, shapes deformation are performed through an action of the group of diffeomorphisms with finite regularity ($C^k$ or Sobolev). Following Lie group ideas, small deformations can be described as vector fields on $\R^d$, representing at each point the directions and speed guiding the evolution of shapes. This group of diffeomorphisms can be enriched by another finite-dimensional group $G$ that can also act on shapes, to define constrained deformations as rotations, translations or scalings.
In a recent work \cite{gga,Pierron2024}, authors suggest a general geometrical approach to define actions of infinite dimensional differentiable groups (namely half-Lie groups) on shape spaces. We adopt this framework and define hypotheses on the group $G$ in order to define the group of deformations as a semidirect product of $G$ with the group of diffeomorphisms. We study here the differential properties of this new group of deformations and define a sub-Riemannian structure on it.

\subsubsection{First definitions and differentiable structure}\label{first_def}
Let $G$ be a finite-dimensional Lie group and $\mathfrak{g}$ be its Lie algebra. The purpose of this part is to establish the theoretical framework for coupling the Lie group $G$ with the group of $C^k-$diffeomorphisms that vanishes to infinity and whose derivatives also vanish as infinity

 \[
\Diff_{C_0^k}(\R^d) = \left(\id+C_0^k(\R^d,\R^d)\right)\cap \Diff^1(\R^d)
\]
We will first make some assumptions to define this coupling as a semidirect product of $G$ and $\Diff_{C_0^k}(\R^d)$ along with its differentiable structure compatible with the composition law.
We consider the following hypotheses:
\begin{description}[leftmargin=*]
    \item[Action of $G$ on $\R^d$] We suppose $G$ acts smoothly via diffeomorphisms on $\R^d$ and we denote $g \cdot x$ the action of $g \in G$ on $x \in \R^d$. 
    \item[Action of $G$ on $\Diff_{C_0^k}(\R^d)$] We suppose that the action of $G$ on $\R^d$ can be lifted to a continuous right action through smooth automorphisms on $\operatorname{Diff}_{C_0^k}(\R^d)$. We denote by $\rho_g(\varphi)$ the action of $g\in G$ on $\varphi \in \operatorname{Diff}_{C_0^k}(\R^d)$. This means that we have a continuous mapping $\rho : G \to \operatorname{Aut}(\operatorname{Diff}_{C_0^k}(\R^d))$, such that for all $\varphi,\psi \in \operatorname{Diff}_{C_0^k}(\R^d) \text{ and } g,h\in G$, we have
\begin{align}
&\rho_{gh} (\varphi)= \rho_h(\rho_g(\varphi)), \label{right_action_hyp} \\
&\rho_g (\varphi\circ \psi) = \rho_g(\varphi)\circ \rho_g(\psi), \label{aut_hyp} \\
&g^{-1} \cdot \varphi(g\cdot x) = \rho_g (\varphi)(x),\quad \forall x\in\R^d
\label{comp_hyp}
\end{align}
and that the mapping $\varphi \mapsto \rho_g(\varphi)$ is a smooth diffeomorphism on $\operatorname{Diff}_{C_0^k}(\R^d)$. The third condition \eqref{comp_hyp} is the compatibility condition between the action of $G$ on $\R^d$ and the action of $G$ on $\Diff_{C_0^k}(\R^d)$.

\item[Regularity of the action] For $u\in T_{\id}\operatorname{Diff}_{C_0^k}(\R^d)=C_0^k(\R^d,\R^d)$, since the action $\rho_g$ is differentiable, we can consider  $d_{\id}\rho_g u = \partial_\varphi \vert_{\varphi=\id}\rho_g(\varphi) \cdot u$ the infinitesimal action of $g$ on $u$. Finally, we suppose that the action $\rho$ induces a continuous action on $\operatorname{Diff}_{C_0^{k+l}}(\R^d)$ by smooth diffeomorphisms, and such that 
$$
\begin{array}{ccl}
    G\times \operatorname{Diff}_{C_0^{k+l}}(\R^d) & \longrightarrow & \operatorname{Diff}_{C_0^k}(\R^d)  \\
     g,\varphi & \longmapsto  & \rho_g(\varphi)
\end{array}
$$ is $C^l$. 
\end{description}
Under those assumptions, we can define the semidirect product :
$$\mathcal{G}^k = G\ltimes\operatorname{Diff}_{C_0^k}(\R^d)$$  with composition law
\begin{equation}
    (g,\varphi)(g',\varphi') = (gg',\rho_{g'}(\varphi)\circ\varphi')
\end{equation}

\begin{rem}[Inclusion in $\Diff_{C^k}(\R^d)$]
\label{incl_Diff}
The conditions we stated before are quite natural since the semidirect product $\mathcal{G}^k$ is a subgroup of the group $\Diff_{C^k}(\R^d)$ of all $C^k$ diffeomorphisms of $\R^d$. Indeed, the first hypothesis of the action of $G$ on $\R^d$ can be restated as a group inclusion
\[
\app{G}{\Diff_{C^k}(\R^d)}{g}{[\varphi_g :x\mapsto g\cdot x]}
\]
This allows to define the injection
\[
i:\left\{\app{\mathcal{G}^k}{\Diff_{C^k}(\R^d)}{g,\varphi}{[\varphi_g\circ\varphi :x\mapsto g\cdot \varphi(x)]}\right.
\]
Note that this injection is also a group morphism, since 
\begin{align*}
    i\big((g,\varphi)(g',\varphi')\big)(x)&=i(gg',\rho_{g'}(\varphi)\circ \varphi')(x)\\
    &= gg'\cdot \left(\rho_{g'}(\varphi)\circ \varphi'(x)\right) \\
    &= g\cdot \varphi(g'\cdot\varphi'(x)) \\
    &=i(g,\varphi)\circ i(g',\varphi')(x).
\end{align*}

In most applications, the action of $\mathcal{G}^k$ on shape spaces will be induced by the action of $\Diff_{C^k}(\R^d)$ on these shape spaces as presented in the next sections. However, this description as a semidirect product allows us to separate both the action of $G$ and the action of the group of diffeomorphisms $\Diff_{C_0^k}(\R^d)$.
\end{rem}

The natural concept of a group endowed with a differentiable Banach structure compatible with the composition law is called a Banach right half-Lie group \cite{bauer2023regularity}.
Before proving that the group $\mathcal{G}^k$ acquires such a structure, we first recall its definition.
\begin{defi}[Half-Lie group]A Banach (right) half-Lie group is a topological group with a smooth
Banach manifold structure, such that the right translation $R_{g'} : g \to gg'$ is smooth.  For $l\in \mathbb{N}$, its set of $C^k$-elements is the subgroup of element $g$ such that the left translation $L_g:g'\mapsto gg'$ and $L_{g^{-1}}: g'\mapsto g^{-1}g'$ are $C^l$.
\end{defi}

\begin{rem}
    In a sense a half-Lie group behaves like a Lie group on the \emph{left}, but not on the \emph{right} where we are only assured to be continuous and not differentiable. 
\end{rem}

In particular, the group of diffeomorphisms $\Diff_{C_0^k(\R^d)}$ is an example of Banach half-Lie group \cite{bauer2023regularity,ARGUILLERE2015139}. Then, using the regularity conditions of the action of $G$ on the group $\Diff_{C_0^k}(\R^d)$, we prove that the group $\mathcal{G}^k$ is a Banach half-Lie group, and some regularity properties of the multiplication and the inverse in the group. 

\begin{prop}[Differential Structure of $\mathcal{G}^k$]
\label{gga_struc_G}
    The group $\mathcal{G}^k$ is a Banach right half-Lie group. 
Furthermore the subgroup $\mathcal{G}^{k+l}$ is exactly the group of $C^l$-differentiable elements of $\mathcal{G}^k$, and we get the following regularity properties
\begin{description}[leftmargin=*]
\label{HypGroup}
    \item[(G.1)] \label{G.1} $\mathcal{G}^{k+1}$ is a subgroup of $\mathcal{G}^k$ with smooth inclusion.
    \item[(G.2)] For $l\geq 0$, the inverse mapping $\operatorname{inv}:\mathcal{G}^{k+l}\to\mathcal{G}^k$ is $C^l$.
    \item[(G.3)] \label{G.2} For $l\geq 0$, the induced multiplication
    \[
 \begin{array}{ccl}
    \mathcal{G}^{k+l}\times \mathcal{G}^k & \longrightarrow & \mathcal{G}^k \\
    ((g,\varphi),(g',\varphi')) &\longmapsto & (gg',\rho_{g'}(\varphi)\circ\varphi')\\
  \end{array}
    \]
is $C^l$ and $C^{\infty}$ in the first variable $g'$ for $g$ fixed.
    \item[(G.4)] \label{G.3} For $l\geq 0$, the induced right infinitesimal translation
    \[
 \begin{array}{ccl} 
    T_{(e_G,\id)}\mathcal{G}^{k+l}\times \mathcal{G}^k & \longrightarrow & T\mathcal{G}^k \\
    ((X,u),(g,\varphi)) &\longmapsto & (X,u)\cdot (g,\varphi) =(d_{e_G}R_g(X),d_{e_G}\rho_g (u) \circ \varphi) \\
  \end{array}
\] 
is a $C^l$ mapping, and $C^{\infty}$ with regards to the first variable. 
\item[(G.5)] The induced left infinitesimal translation
    \[ 
 \begin{array}{ccl} 
    \mathcal{G}^{k+1}\times T_{(e_G,\id)}\mathcal{G}^k & \longrightarrow & T\mathcal{G}^k \\
    ((g,\varphi),(X,u)) &\longmapsto & d_{(e_G,\id)}L_{(g,\varphi)}(X,u)\\
  \end{array}
\] is $C^1$.
\end{description}
\end{prop}
\begin{rem}
These regularity properties (G.1-5) were also used and stated in \cite{gga}, as a general setting to extend large deformation model for registration. 
\end{rem}
\begin{proof}
The group $\mathcal{G}^k$ is immediately a topological group and a Banach manifold for $k\geq1$. It remains to prove that right translations in the group are smooth. Since $G$ is a Lie group, it is sufficient to prove that for any $(h,\psi)\in G\ltimes \operatorname{Diff}_{C^k_0}(\R^d,\R^d)$, the mapping
$$
\begin{array}{ccl}
     \operatorname{Diff}_{C_0^{k}}(\R^d) & \longrightarrow & \operatorname{Diff}_{C_0^k}(\R^d)  \\
     \varphi & \longmapsto  & \rho_h(\varphi)\circ\psi
\end{array}
$$ is smooth. This follows from the fact that $\rho$ is smooth and right translations in $\operatorname{Diff}_{C_0^k}(\R^d)$ are smooth since it is a half-Lie group. Therefore the group $\mathcal{G}^k$ is a Banach half-Lie group. We prove now the properties (G.1-5). Moreover, the group $\Diff_{C_0^{k+l}}(\R^d)$ is a subgroup of $\Diff_{C_0^k}(\R^d)$ with smooth inclusion, and thus (G.1) follows. We then prove that the subgroup \[
(\mathcal{G}^k)^l = \{(g,\varphi)\in \mathcal{G}^k, \ L_{(g,\varphi)} \text{ is } C^l\}
\]of $C^l$-differentiable elements of $\mathcal{G}^k$ as defined in \cite{bauer2023regularity} is exactly $\mathcal{G}^{k+l}$. Since $G$ is a Lie group (and therefore $G^l=G$), and the space of $C^l$-elements of $\operatorname{Diff}_{C_0^k}(\R^d)$ is $\operatorname{Diff}_{C_0^{k+l}}(\R^d)$ \cite{bauer2023regularity}, then the space $(\mathcal{G}^k)^l$ is simply
$$
(\mathcal{G}^k)^l = \{(g,\varphi) \in \mathcal{G}^k,\ \mbox{s.t. } \forall h\in G,\, \rho_h(\varphi)\in \operatorname{Diff}_{C_0^{k+l}}(\R^d)\}.
$$ Therefore if $(g,\varphi)\in (\mathcal{G}^k)^l$, by taking $h=e_G$, we immediately have $\varphi\in\operatorname{Diff}_{C_0^{k+l}}(\R^d)$. Conversely if $\varphi\in\operatorname{Diff}_{C_0^{k+l}}(\R^d)$, by hypothesis, $\rho_h(\varphi)$ is also in $\operatorname{Diff}_{C_0^{k+l}}$ for any $h\in G$. In particular, (G.3) follows, and by implicit function theorem, (G.2) is also proved as a consequence of (G.3). Moreover the tangent space at identity of $\mathcal{G}^k$ is given by 
\[
T_{(e_G,\id)}\mathcal{G}^k= C_0^k(\R^d,\R^d)\oplus \mathfrak{g}
\]
where $\mathfrak{g}$ denotes the Lie algebra of the group $G$. By smoothness of $d_{e_G}R$ in the Lie group $G$, and the differentiable conditions on the map $\rho:G\times\Diff_{C_0^k}(\R^d)\to\Diff_{C_0^k}(\R^d)$, (G.4) also follows. Finally, for $(g,\varphi)\in\mathcal{G}^{l+1}$, the differential of the left multiplication is given by
\[
d_{(e_G,\id)}L(g,\varphi)(X,u)=\left(d_{e_G}L_g(X), d\left(\rho_g(\varphi)\right)u\right),
\] where $d\left(\rho_g(\varphi)\right)u : x \mapsto d_x(\rho_g(\varphi))u(x)$ is $C^k$, so that (G.5) is also proved.
\end{proof}
\begin{rem}
\label{rem:loc_add}
Note that if the action of $G$ on $\R^d$ is proper, the group $\mathcal{G}$ can be equipped with a right-invariant local addition. Indeed, since $G$ is a finite-dimensional Lie group, it can be endowed with a right-invariant Riemannian metric and therefore with a right-invariant local addition through the exponential map. We denote this local addition by $\tau_G : V_G\subset TG \to G$, where $V_G$ is an open neighborhood of the zero section of $TG$. Moreover, since the group $G$ acts properly on $\R^d$, there exists a Riemannian metric on $\R^d$ that is $G$-invariant \cite[theorem 2]{koszul1965lectures}. Therefore the exponential map $\exp$ of this metric is a local addition on $\R^d$ that is equivariant with regards to the action of $G$, i.e. for $g\in G, (x,v)\in \R^d\times\R^d$, we have
$$
\exp_{g\cdot x}(g\cdot v) = g\cdot \exp_x(v)
$$
Now following \cite{bauer2023regularity}, this induces a right-invariant local addition $\tau_{\operatorname{Diff}}$ on $\operatorname{Diff}_{C_0^k}(\R^d)$ by taking the push-forward \[
\exp_*: \app{T\Diff_{C_0^k}(\R^d)}{\Diff_{C_0^k}(\R^d)}{(\varphi,v\circ\varphi)}{\exp_{\varphi}(u\circ\varphi)}, 
\] and because of the compatibility condition \eqref{comp_hyp}, this local addition is also equivariant with regards to the action of $G$. One can easily verify that the product $\tau_G\times\tau_{\operatorname{Diff}}$ thus defines a right-invariant local addition on $\mathcal{G}^k$. Therefore, under this condition, by \cite[theorem 3.4]{bauer2023regularity}, the regularity conditions (G.1-5) directly follows. Example \ref{ex_anisotropic_scaling} gives an example of group $G$ that do not act properly on $\R^d$.
\end{rem}

Those properties on the groups $\mathcal{G}^k$ allow to perform differentiable calculus. In particular, as stated in Prop. \ref{exist_flow}, the regularity of the group $\mathcal{G}^k$, namely whether curves of its tangent space at identity can be integrated to curves in the half-Lie group, results from those properties. Indeed, it allows to transport the dynamics from  $\mathcal{G}^k$ to $T_{(e_G,\id)}\mathcal{G}^k$ via the evolution equation \begin{equation*}
   (\dot{g}_t,\dot{\varphi}_t) = d_{(e_G,\id)}R_{(g_t,\varphi_t)}(X_t,u_t), \quad
    (g_0,\varphi_0)=(e_G,\id)
\end{equation*}
with $(X_t,u_t)\in L^2([0,1],T_{(e_G,\id)}\mathcal{G}^k)$ the Eulerian derivative. This approach  enables calculus to be performed in a vector space.  The regularity of Banach half-Lie groups was proved in \cite[Theorem 4.2]{bauer2023regularity} when restricting to smooth curves in the tangent space of identity. In our setting, we consider less regular curves with $L^2$-regularity,  which generates absolutely continuous curves in $\mathcal{G}^k$. We denote $AC_{L^2}([0,1],\mathcal{G}^k)$ the space of absolutely continuous curves in $\mathcal{G}^k$ (cf. \cite{glo} for a complete exposition) and we state the following result from \cite{gga}.

\begin{prop}[{Existence and uniqueness of a flow in $\mathcal{G}^k$ \cite[Prop~2.6]{gga}}]\label{exist_flow}
    Let $(X_t,u_t) \in L^2([0,1],\mathfrak{g}\times C_0^{k+1}(\R^d,\R^d))$ be time-varying vector fields. There exists a unique solution $(g_t,\varphi_t)\in AC_{L^2}([0,1],\mathcal{G}^k)$, called the flow of $(X_t,u_t)$, to the system 
    
\begin{equation}
\label{evol_G}
      \left\{
    \begin{array}{l}
   (\dot{g}_t,\dot{\varphi}_t) = d_{(e_G,\id)} R_{(g_t,\varphi_t)}(X_t,u_t) \\
    (g_0,\varphi_0)=(e_G,\id)
    \end{array} 
    \right.
\end{equation}    
\end{prop}

We conclude this section with some examples that will be further developed in the following sections.
\begin{exam}[Isometries and diffeomorphisms]
     Consider the group $\operatorname{Isom}(\R^d)\coloneq  \operatorname{SO}_d \ltimes \R^d$ of isometries of $\R^d$. It naturally acts by conjugation on the group of diffeomorphisms by
\[
((R,T)\cdot \varphi)(x) = R^\top\varphi(Rx+T)-R^\top T.
\]
This allows to define the semidirect product $\operatorname{Isom}(\R^d)\ltimes \operatorname{Diff}_{C_0^k}(\R^d)$.
\end{exam}
\begin{exam}[Anisotropic scalings and diffeomorphisms]\label{ex_anisotropic_scaling}
     For $\rho=(\rho_1,...,\rho_d)\in\R_{>0}^d$, we denote $\D{\rho}\in M_d(\R)$ the diagonal matrix with coefficients $(\rho_i)_i$
\[
\D{\rho} =\begin{bmatrix}
\rho_1 & 0        & \cdots & 0 \\
0        & \rho_2 & \cdots & 0 \\
\vdots   & \vdots   & \ddots & \vdots \\
0        & 0        & \cdots & \rho_d
\end{bmatrix}.\]
The scaling group $(\R_{>0})^d$ also acts by conjugation on the diffeomorphism group by
\[
(\rho\cdot \varphi)(x) = \D{\rho^{-1}}\varphi(\D{\rho}x).
\] allowing to define the semidirect product group of anisotropic scalings and diffeomorphisms $\R^d_{>0}\ltimes \Diff_{C_0^k}(\R^d)$. Note that in this example, the action of the group of anisotropic scalings $\R^d_{>0}$ on $\R^d$ is not proper, since the stabilizer of $0\in\R^d$ 
\[
\operatorname{Stab}_{\R^d_{>0}}(0)=\{\rho\in\R^d, \ \text{s.t. }\rho\cdot 0 =0\}
\] is the whole group $\R^d_{>0}$, which is not compact. In particular, we cannot use the construction of remark \ref{rem:loc_add}.
\end{exam}
\subsubsection{Sub-Riemannian metric on \texorpdfstring{$\mathcal{G}^k$}{SR}}
In this section, we introduce right-invariant metrics on the group $\mathcal{G}^k$ following the general framework in \cite{gga}, \cite{Pierron2024}. Indeed, proposition \ref{gga_struc_G} places our setting exactly within this framework, from which the completeness of these metrics follow.
Suppose the algebra $\mathfrak g$ is equipped with a scalar product $\langle\cdot,\cdot\rangle_{\mathfrak{g}}$, so that we define the natural right-invariant metric :
\begin{equation}
    \langle X,Y \rangle_g = \langle Xg^{-1},Yg^{-1}\rangle_{\mathfrak{g}}
\end{equation}
Let also $V$ be a Hilbert space continuously embedded in $C_0^{k+2}(\R^d,\R^d)=T_{\id} \operatorname{Diff}_{C_0^{k+2}}(\R^d)$ and the map $K_V : V^* \to V$ denotes the inverse of the Riesz isometry on $V$. This defines a right-invariant sub-Riemannian structure on $\mathcal{G}^k$ defined by
\begin{itemize}
    \item the vector bundle $\mathfrak{g}\times V$ on $\mathcal{G}^k$, that is equivalent to the bundle $\Delta_{g,\varphi}\coloneqq T_gG\times (d_{\id}\rho_g (V))\circ\varphi \hookrightarrow T_{g,\varphi}\mathcal{G}^k$
    \item the morphism $d_{(e_G,\id)}R_{(g,\varphi)} :\left\{\begin{array}{ccl}
     \mathfrak{g}\times V &\longrightarrow  &\Delta_{g,\varphi}   \\
     X,u &\longmapsto & (d_{e_G}R_g(X), (d_{\id}\rho_g (u) )\circ \varphi)
\end{array}\right.$
    \item the metric $\langle (X,u) ,\ (X,u)  \rangle = \langle X,X\rangle_{\mathfrak{g}} + \langle u,u\rangle_{V}$
\end{itemize}
Let $L\left((g_t,\varphi_t),(X_t,u_t)\right)$ denotes the usual sub-Riemannian length given by
\begin{equation*}
L\left((g_t,\varphi_t),(X_t,u_t)\right) = \int_0^1 \sqrt{\lvert X_t \rvert_{\mathfrak{g}}^2+\lvert u_t\rvert_V^2} \, dt    
\end{equation*}
 We can equip $\mathcal{G}^k$ with the associated sub-Riemannian distance $d_{SR}$ defined as in \cite{arguillère_trélat_2017,general_setting,gga}
\begin{equation*}
    d_{SR}\left((g,\varphi),(g',\varphi')\right) \coloneqq \inf \{L\left((g_t,\varphi_t),(X_t,u_t)\right)  \mid \left((g_t,\varphi_t),(X_t,u_t)\right)\in \operatorname{Hor}_{L^1}\left((g,\varphi),(g',\varphi')\right)\}.
\end{equation*}
where $\operatorname{Hor}_{L^1}\left((g,\varphi),(g',\varphi')\right)$ is the set of $L^1$ horizontal systems joining $(g,\varphi)$ to $(g',\varphi')$, i.e curves $\left((g_t,\varphi_t),(X_t,u_t)\right)$ with $(X_t,u_t)\in L^1([0,1],V),\, (g_t,\varphi_t)\in AC_{L^1}([0,1],\mathcal{G}^k)$ such that for all $t\in [0,1]$, $(\dot{g}_t,\dot{\varphi}_t)=d_{(e_G,\id)}R_{(g_t,\varphi_t)}(X_t,u_t)$. Equivalently, the sub-Riemannian distance can be obtained by minimizing the energy \cite{ARGUILLERE2015139,general_setting,gga}
$$
E\left((g_t,\varphi_t),(X_t,u_t)\right) = \int_0^1 \lvert X_t \rvert_{\mathfrak{g}}^2+\lvert u_t\rvert_V^2 \, dt.
$$
We will rely on this expression to determine minimizers.
The space $(\mathcal{G}^k,d_{SR})$ is therefore a metric space and we have the following completeness properties as a consequence of \cite[Theorem 3.12]{gga}.
\begin{prop}[Completeness of $\mathcal{G}^k$] We get the following
\begin{enumerate}
    \item The space $(\mathcal{G}^k,d_{SR})$ is metrically complete.
    \item \label{Mult:min_ex} Suppose $V$ is $G$-invariant, in the sens that $d_{\id}\rho_g$ is an isometry of $V$ for all $g\in G$. Then the space $(\mathcal{G}^k,d_{SR})$ is geodesically convex, meaning that for any $(g,\varphi),(g',\varphi')\in \mathcal{G}^k$ such that $d_{SR}\left((g,\varphi),(g',\varphi')\right) < \infty$, there exists a minimizing geodesic connecting $(g,\varphi)$ and $(g',\varphi')$
\end{enumerate}
\end{prop}
\begin{proof}
    First point is already proved in \cite[Theorem 3.12]{gga}.  We now prove the second point.
    Since for any $g\in G$, the mapping $d_{\id}\rho_g$ is an isometry of $V$, then we get that for any horizontal system $\left((g_t,\varphi_t),(X_t,u_t)\right)$,
    \begin{align*}
        L\left((g_t,\varphi_t),(X_t,u_t)\right) &= \int_0^1 \sqrt{\lvert X_t \rvert_{\mathfrak{g}}^2+\lvert u_t\rvert_V^2} dt \\
        &= \int_0^1 \sqrt{\lvert X_t \rvert_{\mathfrak{g}}^2+\lvert d_{\id}\rho_{g_t}(u_t)\rvert_V^2} dt \\
        &=  L\left((g_t,\varphi_t),(X_t, d_{\id}\rho_{g_t}(u_t))\right)
    \end{align*}
    Therefore, we consider the change of variable $\tilde{u}=d_{\id}\rho_{g}(u)
    $ and we get that 
    \begin{eqnarray}
    d_{SR}\left((g,\varphi),(g',\varphi')\right) &= &\inf_{(X,\tilde{u}) \in L^2([0,1],\mathfrak{g} \times V)} L\left((g_t,\varphi_t),(X_t,\tilde{u}_t)\right)   \\
     \text{ s.t }& &     \left\{
        \begin{array}{ll}
        \label{evol_G2}
            \dot{g}_t &= d_{e_G}R_{g_t}(X_t)\\
            \dot{\varphi}_t &= \tilde{u}_t\circ\varphi_t            
        \end{array} 
        \right.
\end{eqnarray}
   We define the  endpoint mapping $\operatorname{End}:L^2(I,\mathfrak{g}\times V)\to \mathcal{G}^k$ such that for any $(X,u)\in L^2([0,1],\mathfrak{g}\times V)$, $\operatorname{End}(X,u)=(g_1,\varphi_1)$ where the curve $(g_t,\varphi_t)$ satisfies the dynamic in \eqref{evol_G2}. To prove geodesic convexity, it suffices to prove that the endpoint mapping is weakly continuous \cite[Theorem 3.12]{gga}. This holds since $G$ is a (finite dimensional) Lie group and  by \cite{articleTro1995,Trouve1998,BegMillerTrouveYounes2005} the endpoint mapping $u\in L^2([0,1],V)\to \varphi^u_1$ is weakly continuous, where $\varphi_t^u$ satisfies
    \begin{equation*}
            \dot{\varphi}^u_t=u_t\circ\varphi_t^u, \quad \varphi_0^u=\id
        \end{equation*}
\end{proof}

We finish by a short discussion on the characterisation of the geodesics, and more particularly the critical points of the energy. In infinite dimension, as described by Arguillère and Trélat (2017)\cite{arguillère_trélat_2017,gga}, there are three types of sub-Riemannian geodesics: normal, abnormal and strictly abnormal. We will only focus here on the sub-Riemannian normal geodesics as they also correspond to the critical points for the inexact matching problem we will consider in this paper (cf. appendix \ref{app:critic_points}). An horizontal system $(g,\varphi,X,u)\in\operatorname{Hor}_{L^1}$ is said to be a \emph{normal geodesic} if there exists non-zero Lagrange multipliers $(\lambda,p^g,p^\varphi)\in\R\times T_{(g_1,\varphi_1)}^*\mathcal{G}^k$ such that 
\begin{equation}
    \lambda dE(X,u) + d\operatorname{End}(X,u)^*(p^g,p^\varphi)=0
\end{equation}
The sub-Riemannian structure on $\mathcal{G}^k$ defines a co-metric  $K : T^*\mathcal{G}^k\to T\mathcal{G}^k$ given by
\[
K_{(g,\varphi)}(p^g,p^\varphi) = \left((d_{e_G}R_g)K_\mathfrak{g} (d_{e_G}R_g)^*p^g,(d_{\id}R_\varphi )(d_{\id}\rho_g)K_V (d_{\id}\rho_g)^*(d_{\id}R_\varphi)^* p^\varphi\right).
\]
and the corresponding Hamiltonian given by
\begin{align*}
    H\left(g,\varphi,p^g,p^\varphi\right) &= \frac{1}{2}\left((p^g,p^\varphi)\mid K_{(g,\varphi)}(p^g,p^\varphi) \right) \\
    &= \frac{1}{2}(\lvert K_\mathfrak{g} (d_{e_G}R_g)^*p^g\rvert^2_{\mathfrak{g}} +\lvert K_V (d_{\id}\rho_g)^*(d_{\id}R_\varphi)^*p^\varphi\rvert^2_V)
\end{align*}
This Hamiltonian is the energy expressed in the cotangent variables.
It drives the geodesic equations in Hamiltonian form. We recall the following result about the characterisation of sub-Riemanian geodesics.

\begin{prop}[{Sub-Riemannian geodesics, \cite[Theorem~3.15]{gga}}]
    Let $(g_t,\varphi_t)\in \operatorname{Hor}_{L^1}$ a horizontal system. Then $(g_t,\varphi_t,X_t,u_t)$ is a normal geodesic if and only if there exists a covector $(p_t^g,p_t^\varphi)\in T_{(g_t,\varphi_t)}\mathcal{G}^k$ such that 
    \begin{equation}
        (\dot{g}_t,\dot{\varphi}_t)=\nabla^\omega H(g_t,\varphi_t,p^g_t,p^\varphi_t).
    \end{equation}
\end{prop}

Note that the property of $V$ being $G$-invariant, that is $d_{\id}\rho_gV =V$ for all $g\in G$, and that $d_{\id}\rho_g$ is an isometry of $V$, is related to whether $G$ acts properly on $\R^d$, since the construction of such Hilbert spaces $V$ relies on the construction of $G$-invariant metrics in $\R^d$. In such case, the Hamiltonian simplifies to
\[
    H\left(g,\varphi,p^g,p^\varphi\right) = \frac{1}{2}(\lvert K_\mathfrak{g} (d_{e_G}R_g)^*p^g\rvert^2_{\mathfrak{g}} +\lvert K_V (d_{\id}R_\varphi)^*p^\varphi\rvert^2_V).
\]
In summary, the sub-Riemannian structure on the semidirect product group $\mathcal{G}^k$ induces a co-metric and an associated Hamiltonian formulation that fully characterize the normal geodesic, that is, the critical points of the energy relevant to the inexact matching problem. These geodesics correspond to trajectories 
$(g_t, \varphi_t)$ whose evolution is governed by the Hamiltonian flow $\nabla^\omega H$. When the space $V$ is $G$-invariant, the Hamiltonian simplifies, leading to a decoupled expression.
\subsection{Shape spaces}
\label{Sec:shape_space}
In this section, we adopt the framework developed in \cite{general_setting} and \cite{gga} to define shape spaces through actions of the groups previously defined on general Banach manifolds. This allow to induce sub-Riemannian structures on shape spaces, with a diffeomorphic part as in the classical LDDMM framework, enriched with deformations generated by a finite-dimensional Lie group $G$ (rigid motions, anisotropic motions, etc.). In particular, we will study how to separate both actions of $G$ and of the diffeomorphism part. Indeed, in many examples, the diffeomorphism part can actually learn the transport induced by $G$, and we would rather want the diffeomorphism to learn only the motion that cannot be performed by $G$. By requiring the metric on the diffeomorphism group to be invariant under the action of $G$, we can perform reduction techniques and decouple the differents modes of motions.

\subsubsection{Action on shape spaces}
\label{Action_Q}
Let $\mathcal{Q}$ be a Banach manifold, representing the shape space. We define then the augmented shape space $\tilde{\mathcal{Q}} = G\times \mathcal{Q}$, where the $G$ part gives a representation of the shape in a coarsest scale. For example, it can represent the position of the shapes in $\R^d$ if $G$ is the group of translations, or the orientation if $G$ is $SO_d$. The augmented shape space $\tilde{\mathcal{Q}}$ allows to keep track of the action of $G$ on the shapes.
We suppose now that the group $\operatorname{Diff}_{C_0^k}(\R^d)$ acts on $\mathcal{Q}$ with following regularity conditions \cite{ArguillereThesis}
\begin{description}[leftmargin=*]
    \label{HypActions}
\item[Continuity of the action] The action $(\varphi,q) \mapsto \varphi\cdot q$ is continuous.
\item[Infinitesimal action] For all $q \in \mathcal{Q}$, the mapping $\varphi \mapsto \varphi\cdot q$ is $C^{\infty}$, and we denote for $u\in C_0^k(\R^d,\R^d)$, $u\cdot q = \partial_\varphi |_{\varphi=\id }(\varphi\cdot q) $ its continuous differential in $\id$, also called the infinitesimal action of $u$ on $q$.
\item[Regularity of the action] For $l>0$, the mappings 
$$
\begin{array}{ccc}
  \begin{array}{ccl}
    \Diff_{C_0^{k+l}}\times \mathcal{Q} & \longrightarrow & \mathcal{Q} \\
    (\varphi, \ q) &\longmapsto & \varphi\cdot q \\
  \end{array}
     & \mbox{ and } & 
  \begin{array}{ccl}
    C_0^{k+l}(\R^d,\R^d)\times \mathcal{Q} & \longrightarrow & T\mathcal{Q} \\
    (u, \ q) &\longmapsto & u\cdot q \\
  \end{array}
\end{array}
$$
are $C^l$.
\end{description}
In addition, assume that $G$ acts smoothly via diffeomorphisms on $\mathcal{Q}$, and  assume the following compatibility condition
\begin{equation}
\label{comp_cond}
    \varphi\cdot(g\cdot q) = g \cdot (\rho_g (\varphi) \cdot q).
\end{equation}
Therefore, we define the group action of $\mathcal{G}^k = G \ltimes \Diff_{C_0^k}(\R^d)$ on the direct product $\tilde{\mathcal{Q}} \coloneqq G \times \mathcal{Q}$ by
\begin{equation}
    A : \app{\mathcal{G}^k \times \tilde{\mathcal{Q}}}{\tilde{\mathcal{Q}}}{(g,\varphi),(h,q)}{(gh,g\cdot (\varphi \cdot q))}
\end{equation}

\begin{rem}[Compatibility condition]
    Note that the compatibility condition \eqref{comp_cond} coincides with condition \eqref{comp_hyp} when $\mathcal{Q} = \mathbb{R}^d$. Therefore condition \eqref{comp_cond} is quite natural and is often a direct consequence of the particular condition \eqref{comp_hyp} when the shape space is one of the classical ones (landmarks, curves, images). Moreover, following remark \ref{incl_Diff}, we know that each pair $(g,\varphi)$ defines a diffeomorphism $\varphi_{g}\circ\varphi\in\Diff_{C^k}(\R^d)$. Hence condition \eqref{comp_cond} often guarantees that the action of $\mathcal{G}^k$ can be induced from the deformation action of $\Diff_{C^k}(\R^d)$.
\end{rem}

We prove in the following proposition that the group $\mathcal{G}^k$ and its action on the manifold $\tilde{\mathcal{Q}}=G \times \mathcal{Q}$ defines a shape space in the sense of \cite{gga}.
\begin{prop}[Shape space $\tilde{Q}$] \label{tilde_q_shape_space}
    The action of the half-Lie group $\mathcal{G}^k$ on $\tilde{\mathcal{Q}}=G \times \mathcal{Q}$ satisfies the following conditions 
    \begin{enumerate}
        \item The action $A : \mathcal{G}^k \times \tilde{\mathcal{Q}}\to\tilde{\mathcal{Q}}$ is continuous.
        \item For $(h,q) \in \mathcal{Q}$, the mapping $A_{(h,q)}:(g,\varphi)\mapsto (gh,g\cdot(\varphi\cdot q))$ is smooth, and we denote $\xi_{(h,q)}=d_{(e_G,\id)}A_{(h,q)}$ its derivative, called the infinitesimal action.
        \item For $l>0$, the mappings \[A : \apps{\mathcal{G}^{k+l} \times \tilde{\mathcal{Q}}}{\tilde{\mathcal{Q}}}{(g,\varphi),(h,q)}{(gh,g\cdot (\varphi \cdot q))} \ \ \text{and} \ \ \xi : \apps{T_{(e_G,\id)}\mathcal{G}^{k+l} \times \tilde{\mathcal{Q}}}{T\tilde{\mathcal{Q}}}{(X,u),(h,q)}{(Xh, X \cdot q + u\cdot q)}\] are $C^l$. 
    \end{enumerate}
\end{prop}

\begin{proof}
    \par The action $A : \mathcal{G}^k \times \tilde{\mathcal{Q}} \rightarrow \tilde{\mathcal{Q}}$ is continuous since the multiplication is continuous in the Lie group $G$ and the groups $G$ and $\Diff_{C_0^k}(\R^d)$ acts continuously on $\mathcal{Q}$. 
    \par For $(h,q) \in \mathcal{Q}$, the mapping $A_{(h,q)} $ is smooth since  the multiplication is smooth in $G$  and the mappings $g \mapsto g\cdot q$ and $\varphi \mapsto \varphi \cdot q$ are smooth by hypothesis.
    \par For $l>0$, the mappings $A : \mathcal{G}^{k+l} \times \tilde{\mathcal{Q}} \to \tilde{\mathcal{Q}}$ and $\xi : T_{(e_G,\id)}\mathcal{G}^{k+l} \times \tilde{\mathcal{Q}} \to \tilde{\mathcal{Q}}$ are $C^l$. Indeed, the group operation in $G$ is smooth, the action of $G$ on $\mathcal{Q}$ is smooth and the mapping $(\varphi,q)\mapsto \varphi \cdot q$  is $C^l$.
\end{proof}

Consequently, under these conditions, the shape space $\tilde{\mathcal{Q}}$ inherits a sub-Riemannian structure from the group $\mathcal{G}^k$ equipped with a right invariant sub-Riemannian metric as presented in \cite{gga}.

\subsubsection{Inexact matching problem}
\label{Sec:var_pb}
Following the classical LDDMM framework \cite{BegMillerTrouveYounes2005}, we consider a Hilbert space $V$ of vector fields with continuous inclusion in $C_0^{k+2}(\R^d,\R^d)$. Here, $V$ is not necessarily $G$-invariant.  As shown in the previous subsection, we can define a strong sub-Riemannian structure on the group $\mathcal{G}^k$ which allows to perform matching on $\mathcal{Q}$ through the action of $\mathcal{G}^k$. Consider a template shape $q_S\in\mathcal{Q}$ and its corresponding shape in the augmented shape space $(e_G,q_S) \in \tilde{\mathcal{Q}}$. We recall that given two time-varying vector fields $(X_t,v_t) \in L^2([0,1],\mathfrak{g} \times V)$, we can generate a flow in $(g_t,\varphi_t)\in\mathcal{G}^k$ which defines a trajectory of shape $(h_t,q_t):=(g_t,\varphi_t)\cdot (e_G,q_S)$. Note that $h_t=g_t$ since we chose $e_G$ as the source for $G$. The dynamic of the deformed shape under the action of $\mathcal{G}^k$ can be deduced by the infinitesimal action :
\begin{equation}
    \label{dyn_tQ}
    (\dot{g}_t,\dot{q}_t)= \xi_{(g_t,q_t)}(X_t,v_t)=(X_tg_t,X_t\cdot q_t + v_t\cdot q_t), \quad (g_0,q_0) = (e_G,q_S)
\end{equation}
The matching of a source shape $(e_G,q_S)$ onto a target $(g_T,q_T)$ can be expressed as an energy minimization problem :

\begin{eqnarray}
    \label{var_pb}
    \inf_{(X,v) \in L^2([0,1],\mathfrak{g} \times V)} J(X,v) &=& \int_0^1 \frac{1}{2} \vert X_t \vert_{\mathfrak{g}}^2 + \frac{1}{2} \vert v_t \vert_V^2 \, dt + \mathcal{D}(g_1,q_1) \\
     \text{ s.t }& &     \left\{
        \begin{array}{l}
            \dot{g}_t = X_tg_t \\
            \dot{q}_t = X_t \cdot q_t + v_t \cdot q_t \\
            (g_0,q_0) = (e_G,q_S)        
        \end{array} 
        \right.\notag
\end{eqnarray}
where $\mathcal{D} : \tilde{\mathcal{Q}} \rightarrow \R$ is a data attachment term measuring the distance of the deformed shape at final time to the target shape. We will discuss in the next section several examples of data attachment terms. We recall the following result on existence of minimizers \cite[Prop. 3.16]{gga} (cf. also appendix \ref{App:opt_cont} for a discussion on variational problems on Banach manifold).
\begin{prop}[Existence of minimizers of $J$]
Suppose the data attachment term $\mathcal{D}$ is continuous, then the problem \eqref{var_pb} admits minimizers.
\end{prop}
The end of this section is devoted to the computation and the characterization of the critical points of the energy through an Hamiltonian formulation, similarly to the LDDMM framework. We define the following Hamiltonian

\[H(g,q,p^g,p,X,v) = (p^g\,|\, Xg) + (p \,|\, v\cdot q + X\cdot q) - \frac{1}{2} \vert X \vert_{\mathfrak{g}}^2 - \frac{1}{2} \vert v \vert_V^2 \]
By \cite{ArguillereThesis,gga}, critical points of the energy $J$ satisfies the Hamiltonian equations.
\begin{prop}[Critical points of energy and Hamiltonian dynamics]
\label{critical_point_H}
    Let $(X,v)\in L^2([0,1],\mathfrak{g}\times V)$, and denote $(g_t,q_t)$ the associated curve in $\tilde{\mathcal{Q}}$ satisfying the evolution equation \eqref{dyn_tQ}. Then $(X,v)$ is a critical point of the energy problem \eqref{var_pb} if and only if $(g_t,q_t)$ satisfies the Hamiltonian dynamic
    \[
    \left\{
        \begin{array}{ll}          (\dot{g}_t,\dot{q}_t,\dot{p}_t^{g},\dot{p}_t) = \nabla^\omega H(g_t,q_t,p_t^{g},p_t,X_t,v_t) \\
            \partial_{X,v} H(g_t,q_t,p_t^{g},p_t,X_t,v_t) = 0
        \end{array} 
        \right.\notag
\]
where $\nabla^\omega H$ is the partial symplectic gradient of $H$ with regards to the canonical weak symplectic form on $T^*\tilde{\mathcal{Q}}$.
\end{prop}
This result is contained in \cite[Theorem 3.15]{gga} and the proof is recalled in appendix \ref{App:opt_cont}.
\begin{rem}
    In canonical coordinates of the cotangent bundle $T^*\tilde{\mathcal{Q}}= T^*(G\times\mathcal{Q})$, the partial symplectic gradient is given by
    \begin{multline*}
    \nabla^\omega H(g,q,p^{g},p,X,v) = \big(\partial_{p^{g}}H(g,q,p^{g},p,X,v),\partial_pH(g,q,p^{g},p,X,v),\\-\partial_gH(g,q,p^{g},p,X,v), -\partial_qH(g,q,p^{g},p,X,v) \big)
    \end{multline*}
\end{rem}
By denoting respectively $\xi^{\mathfrak{g}}_q$ and $\xi_q$ the infinitesimal actions of $\mathfrak{g}$ and $C_0^k(\R^d,\R^d)$ on $q$, we define the reduced Hamiltonian by

\begin{equation}\label{hamilt_1}
    H(g,q,p^g,p) =\frac{1}{2} \lvert K_{\mathfrak{g}} ((d_{e_G}R_g)^* p^{g} + \xi_q^{\mathfrak{g}*}p) \rvert^2_{\mathfrak{g}} + \frac{1}{2} \lvert K_V \xi_{q}^{*}p \rvert^2_V
\end{equation}
so that the critical points of the energy $J$ are equivalently expressed as the Hamiltonian flow of $H$
\begin{equation*}
    (\dot{g},\dot{q},\dot{p}^{g},\dot{p}) = \nabla^\omega H(g,q,p^{g},p)
\end{equation*}
We can notice that the control $X = K_{\mathfrak{g}} ((d_{e_G}R_g)^* p^{g} + \xi_q^{\mathfrak{g}*}p)$ is parameterized by two different covectors which complicates its interpretation. Indeed, in computational anatomy, the covector is a covariable allowing to perform some statistics on the studied deformation. In this framework, we consider two different deformations, so we would like to be able to measure the contribution of each type of deformations separately, which can be done by having a separate covariable per deformation. To do so, we introduce a new variable
\begin{equation}
\label{goat_change_variable}
    \tilde{q}=g^{-1}\cdot q
\end{equation}
This new variable represents the shape in its intrinsic frame of reference, allowing us to track the non-affine deformation before it is affinely transported. Moreover, in the next section, we will focus on reduction which is easier to do considering this change of variable. The idea is that we will remove the influence of $G$ on the shape, in order to only capture the deformation generated by the diffeomorphism. The dynamic of this new shape is given by 

\begin{equation}
    \dot{\tilde{q}}_t = d_{\id}\rho_{g_t}(v_t) \cdot \tilde{q}_t
\end{equation}
By denoting $\tilde{v}=d_{\id}\rho_{g_t}(v_t) \in \tilde{V}=d_{\id} \rho_{g_t}(V)$, the dynamic can be simply expressed through the infinitesimal action $\dot{\tilde{q}}_t = \xi_{\tilde{q}_t}(\tilde{v}_t) = \tilde{v} _t \cdot \tilde{q}_t$. The evolution of the shape $\tilde{q}_t$ no longer depends on $X$, which is consistent with the idea of removing the influence of the deformation induced by $G$ on the deformed shape. However, $\tilde{q}_t$ depends on the shape $g$ at time $t$. Then, the energy associated to the matching problem can be expressed with respect to this new shape.

\begin{eqnarray} 
    \label{var_pb2}
    \inf_ {(X,v) \in L^2([0,1],\mathfrak{g} \times V)} \tilde{J}(X,v) &=& \int_0^1 \frac{1}{2} \lvert X_t \rvert_{\mathfrak{g}}^2 + \frac{1}{2} \lvert v_t \rvert_V^2  \, dt + \mathcal{D}(g_1,g_1\cdot \tilde{q}_1) \\
     \text{ s.t }& &     \left\{
        \begin{array}{l}
            \dot{g}_t = X_tg_t \\
            \dot{\tilde{q}}_t = \xi_{\tilde{q}_t}(\tilde{v}_t) \\
            (g_0,\tilde{q}_0) = (e_G,q_S)   
        \end{array} 
        \right.\notag
\end{eqnarray}
Similarly to the previous problem \eqref{var_pb}, given that the data attachment term is continuous, we can prove existence of minimizers of problem \eqref{var_pb2}.

\begin{prop}[Existence of the minimizers of $\tilde{J}$]
\label{eq_var_pb}
Suppose the data attachment term $\mathcal{D}$ is continuous, then the problem \eqref{var_pb2} admits minimizers.
\end{prop}
Moreover, the problems \eqref{var_pb} and \eqref{var_pb2} are equivalent through the change of variable $\tilde{q}=g^{-1}\cdot q$
\begin{prop}[Equivalence of the problems $J$ and $\tilde{J}$]

The minimizers of $J$ are exactly the minimizers of $\tilde{J}$.
    
\end{prop}
\begin{proof}
Let \((X^*, v^*) \in L^2(I, \mathfrak{g} \times V)\) be minimizers of \(J\). From the dynamic equations associated to $J$ in \eqref{var_pb}, we can show that $ \tilde{q}_t = d_{\id} \rho_{g_t}(v_t^*) \cdot \tilde{q}_t $. Moreover, since $g_1^{-1} \cdot q_1 = \tilde{q}_1$, we have $J(X^*, v^*) = \tilde{J}(X^*, v^*) $ and thus  $\min J \geq \min \tilde{J}$. Similarly, we can prove the reverse inequality, which leads to the equality of the minimizers.
\end{proof}
In a same fashion, the minimizers of the energy are the geodesics of the Hamiltonian defined by
\[
\tilde{H}(g,\tilde{q},p^g,p',X,v) = (p^g \vert X g) + (p \vert d_{\id}\rho_g(v)\cdot\tilde{q})- \frac{1}{2} \vert X \vert^2_\mathfrak{g} - \frac{1}{2} \vert v \vert^2_V
\]
The reduced Hamiltonian associated is 
\begin{equation} \label{eq:hamilt_tilde_red}
\tilde{H}(g,\tilde{q},\tilde{p}^g,\tilde{p}) = \frac{1}{2}\vert K_{\mathfrak{g}} (d_{e_G}R_g)^*\tilde{p}^g\vert_\mathfrak{g}^2 + \frac{1}{2} \vert K_V(d_{\id}\rho_g)^*\xi_{\tilde{q}}^{*} \tilde{p}\vert_V^2
\end{equation}
We end this section by showing that the two Hamiltonian formulations, $H$ and $\tilde{H}$, are also equivalent.
\begin{prop}[Influence of the change of variable on covectors] \label{relation_moment}
    Let $\xi^\mathfrak{g}$ denote the infinitesimal action of $\mathfrak{g}$ on $\mathcal{Q}$. Let $(p^g_0,p_0),(\tilde{p}^g_0,\tilde{p}_0)\in T_{(e_G,q_S)}^*\tilde{\mathcal{Q}}$ be initial covectors such that $\tilde{p}^{g}_0 = p^{g}_0 + \xi_{q_S}^{\mathfrak{g}*}p_0$ and $\tilde{p}_0 = \big(\partial_q ( e_G \cdot q )\vert_{q=q_S}\big)^*p_0$. Let $(g_t,q_t)$ (resp. $(\tilde{g}_t,\tilde{q}_t)$) be the Hamiltonian flow of $H$ (resp. $\tilde{H}$). Then for all time $t$,
    \begin{equation}
     \left\{
        \begin{array}{l}
            \tilde{g}_t=g_t\\
            \tilde{q}_t= g_t\cdot q_t\\
            \tilde{p}_t = \big(\partial_q ( g_t\cdot q )\vert_{q=\tilde{q}_t}\big)^*p_t\\ 
            \tilde{p}^{g}_t = p^{g}_t + (d_{e_G}R_{g_t^{-1}})^* \xi_{q_t}^{\mathfrak{g}*}p_t
        \end{array} 
        \right.        
    \end{equation}
\end{prop}
\begin{proof}
    The proof is given in appendix \ref{proof_moments}
\end{proof}

\subsubsection{Choice of the data attachment term}

In a registration task, the purpose is to match a template shape $q_S$ onto a target shape $q_T$. A classic data attachment term is a real valued mapping defined over the shape space $\mathcal{Q}$. However, in this framework we consider an enriched shape space $\tilde{\mathcal{Q}} = G \times Q$ so that the data attachment term is defined over both spaces. Covectors are related to the data attachment term by the relation (cf. Theorem \ref{app:critic_points}) \[(p^{g}_1,p_1)=-d\mathcal{D}(g_1,q_1),\] so that the choice of the data attachment term will have consequences on the expression of the geodesics. As detailed below, the data attachment term $\mathcal{D} : \tilde{\mathcal{Q}} \to \R$ can be defined via the data attachment terms on $G$ and $\mathcal{Q}$. 
The choice of data attachment term on $\mathcal{Q}$ depends also on the studied shape space, and is denoted by $\mathcal{D}^{\mathcal{Q}}$. For example, we can choose varifolds \cite{Charon_2013} for landmarks, curves or surfaces, or $L^2$ for images.
As introduced earlier in this section, the finite-dimensional group $G$ represents an information on the shape at a coarse scale, as the orientation or the position of the shape. Consequently, a possible choice of data attachment term $\mathcal{D}^G : G \to \R$ is $\mathcal{D}^G(g) = \mathcal{D}^{\mathcal{Q}}(g\cdot q_S)$. In other words, $\mathcal{D}^G$ only measures the deformation of $G$ to match the source onto the target. We propose the following ideas for the choice of the general term $\mathcal{D}$ :

\begin{itemize}

    \item A natural choice for the data attachment term is to consider one that depends only on the shape $\mathcal{D}(g,q)=\mathcal{D}(q)$ since it is the object we want to match. It follows that $p^{g}_1 = 0$ which implies that $p^{g}_t=0$ for every time $t \in [0,1]$ (or equivalently, after change of variable \eqref{goat_change_variable}, we get for all $t$, that $\tilde{p}_t^g$ is a variable depending on $g_t,\tilde{p}_t,\tilde{q}_t$). This means that we can get rid of the covector $p^g_t$, and that the dynamic of $g_t$ and of $q_t$ is only given by $p_t$ : we do not obtain in that case a decoupling of the deformations.

    \item Another possible choice is to split the data attachment term in two separate data attachment terms, one for the group element and one for the shape $\mathcal{D}(g,q)=\mathcal{D}^G(g) + \mathcal{D}^Q(q)$. Then, $(p_1^{g},p_1)=-(d \mathcal{D}^G(g_1),d\mathcal{D}^G(q_1))$ and the covectors at time $t=1$ are not directly related.

    \item Considering the shape $\tilde{q}= g^{-1} \cdot q$, we can consider a data attachment term that mixes the group element and the shape : $\mathcal{D}(g,\tilde{q})=\mathcal{D}^G(g) + \mathcal{D}^{\mathcal{Q}}(g\cdot \tilde{q})$. Thanks to the change of variable, the data attachment term simplifies to $\mathcal{D}(g,\tilde{q}) = \mathcal{D}^G(g) + \mathcal{D}^{\mathcal{Q}}(q)$.

    \item The data attachment can also be invariant under the action of G, i.e $\mathcal{D}(g,g\cdot \tilde{q})=\mathcal{D}(g,\tilde{q})=\mathcal{D}^G(g) + \mathcal{D}^Q(\tilde{q})$

\end{itemize}

\subsubsection{Invariance and symplectic reduction} \label{reduction}
In this section, we use symmetries in the Hamiltonian in order to decouple the action of $G$ and of the group of diffeomorphisms on shapes. In many case, the deformation induced by the finite dimensional Lie group $G$ can be performed by diffeomorphisms, and therefore we want to avoid the learning by diffeomorphisms of the motion that could be performed only by $G$. We suppose in all this section that the space $V$ is $G$-invariant, meaning that for all $g\in G$, the mapping $T_{\id}\rho_g$ preserves $V$ and is an isometry. This induces symmetries in the Hamiltonian and allows to perform symplectic reduction as in  \cite{MARSDEN1974121,marsden1994introduction,Satzer1977}. Some of the necessary definitions and tools used in this part are also more detailed in the appendix \ref{App:symplecti_geom}.

Let first recall how the symplectic reduction can be performed on the shape space $\mathcal{Q}$ in our setting and some properties of the Marsden-Weinstein symplectic quotient before extending this framework to the augmented shape space $\tilde{\mathcal{Q}} = G \times \mathcal{Q}$. The action of the Lie group G on $\mathcal{Q}$ can be lifted by symplectomorphisms (definition \ref{def:symplecto}) on an action of $G$ on $T^*\mathcal{Q}$ given by
\[
g\cdot (q,p) = (g\cdot q,g^{-1}\cdot^*p)
\]
meaning that for any $h\in T_{g\cdot q}\mathcal{Q}$, we have
\[
\left(g^{-1}\cdot^*p\mid h\right)=\left(p\mid g^{-1}\cdot h\right)
\]
This action is actually a Hamiltonian action (cf. appendix \ref{App:Ham_action} for more details on Hamiltonian action). Indeed, by denoting $\Gamma(T^*\mathcal{Q})$ the space of smooth sections of the cotangent bundle $T^*\mathcal{Q}$, we introduce the infinitesimal action 
\[
\xi^{G}:\app{\mathfrak{g}}{\Gamma(T^*\mathcal{Q})}{X}{\left((q,p)\mapsto  \frac{\partial}{\partial g}\big(g\cdot (q,p)\big)_{|g=e}X\right)},
\] 
and define the mapping $\hat{\mu}:\mathfrak{g}\to C^\infty(T^*\mathcal{Q},\R)$ by 
\begin{equation*}
\hat{\mu}(X)(q,p)= (p\,|\, X\cdot q) .  
\end{equation*} 
The mapping $\hat{\mu}$ defines for any $X\in \mathfrak{g}$, a Hamiltonian function $\hat{\mu}(X) : T^*\mathcal{Q}\to \R$, and the infinitesimal action $\xi^G(X)\in\Gamma(T^*\mathcal{Q})$ gives its Hamiltonian vector field, i.e. 
\[
\nabla^\omega \hat{\mu}(X) = \xi^G(X).
\]
The mapping $\hat{\mu}$ gives also rise to a momentum map $\mu:T^*\mathcal{Q}\to\mathfrak{g}^*$ defined by
$
(\mu(q,p)\,|\, X)= \hat{\mu}(X)(q,p)$.
Moreover, this momentum map is also $G$-equivariant since we have  
\begin{align*}
  (\mu(g\cdot(q,p))\,|\, X) &= \left(p\,|\, g^{-1}\cdot (X\cdot (g\cdot q)) \right) \\
  &= \left(p \,|\, \operatorname{Ad}_{g^{-1}}(X)\cdot q \right) \\
  &= (\operatorname{Ad}_{g^{-1}}^* \mu(q,p)\,|\, X)
\end{align*}
Consequently, the action of $G$ on $T^*Q$ is an Hamiltonian action.
In addition, if we assume that the action of $G$ on $\mathcal{Q}$ is proper and free, we can consider the Marsden–Weinstein quotient $\mu^{-1}(0)/G$ (also sometimes denoted by $T^*\mathcal{Q}//G$) which is a symplectic manifold equipped with the induced weak symplectic form. The notation $T^*\mathcal{Q}//G$ means that to quotient out by the action of $G$, we must follow a 2-step reduction process : we first restrict from $T^* \mathcal{Q}$ to $\mu^{-1}(0)$ and we reduce again via the quotient $\mu^{-1}(0)/G$. Note that in this particular case, we get an identification \cite[Th. 14]{Satzer1977}
\[\mu^{-1}(0)/G \simeq T^*(Q/G).\]
 Moreover, if $V$ is $G$-invariant, then the Hamiltonian $H$ is also $G$-invariant, and therefore, by Noether's theorem, the momentum is conserved on the flow of $H$. \\

We now adapt this framework to the shape space $G \times \mathcal{Q}$. In this setting, we consider that the group $G$ acts only on the second component of $G\times\mathcal{Q}$ as follows
\[
g\cdot (g',q)= (g',g\cdot q)
\]
and we perform, in the same fashion, a symplectic reduction with regards to this action. This action of $G$ on $G \times \mathcal{Q}$ can be lifted by symplectomorphism on $T^* G \oplus T^* \mathcal{Q}$, and will lead to the same momentum map and reduction process. Under the same hypothesis (free and proper action of $G$ on $\mathcal{Q}$), we can deduce that the space 
\[
T^*G\oplus \mu^{-1}(0)/G 
\] defines a symplectic manifold. As stated in the following proposition, a first consequence is that, since the Hamiltonian $H : T^* G \oplus T^* Q \to \R$ defined in \eqref{hamilt_1} is invariant by the action of $G$, it can be reduced to an Hamiltonian on the reduced space $T^*G \oplus \mu^{-1}(0)/G$.

\begin{prop}[Projection of the hamiltonian flow \cite{MARSDEN1974121}]
Let $i : T^*G \oplus \mu^{-1}(0) \hookrightarrow T^*G \oplus T^*\mathcal{Q}$ be the inclusion map and $\pi : T^*G\oplus \mu^{-1}(0)\to T^*G\oplus \mu^{-1}(0)/G$ denote the canonical projection. 
If the Hamiltonian $H$ is invariant under the action of $G$, then the Hamiltonian flow on $T^*G \oplus \mu^{-1}(0)$ associated to the Hamiltonian $H$ induces an Hamiltonian flow on $T^*G \oplus \mu^{-1}(0)/G$ whose Hamiltonian $\bar{H}$ is defined by $\bar{H} \circ \pi = H \circ i$.
\end{prop}

Furthermore, since we get the isomorphism 
\[T^*G \oplus \mu^{-1}(0)/G \simeq T^*G \oplus T^*(\mathcal{Q}/G),\] the Hamiltonian flow of $H$ on $T^*G \oplus \mu^{-1}(0)$ can be projected on an Hamiltonian flow on $T^*G \oplus T^*(\mathcal{Q}/G)$ which can be interpreted as a decoupling between the diffeomorphism part and the action of $G$.

\paragraph{Variational problem and reduction.}We finish this part  by linking this reduction process with the variational problems we are interested in. First, since $V$ is $G$-invariant, for any $v \in V$ we have that $\lvert v\rvert_V = \lvert d_{\id}\rho_{g} v\rvert_V^2$ for all $g\in G$. Therefore the functional \eqref{var_pb2} that we want to minimize becomes :
\[
J(X,v) = \int_0^1 \lvert X_t \rvert_\mathfrak{g}^2 + \lvert v_t\rvert_V^2 \,dt+ \mathcal{D}(g_1,g_1\cdot \tilde{q}_1).
\] 
We start by supposing the the data attachment term $\mathcal{D} : G \times \mathcal{Q} \to \R$ is $G$-invariant with regards to the second variable, i.e. $\forall g'\in G, \  \mathcal{D}(g,g'\cdot q)=\mathcal{D}(g,q)$. In particular the mapping $\mathcal{D}$ induces a reduced data attachment term on $\bar{\mathcal{Q}}=G\times \mathcal{Q}/G$ with associated equivalence class $[q]$, that we denote $\bar{\mathcal{D}}$. In this case minimizing $J$ becomes equivalent to minimizing
\[
\bar{J}(X,v) = \int_0^1 \lvert X_t \rvert_\mathfrak{g}^2 + \lvert v_t\rvert_V^2 \, dt+ \bar{\mathcal{D}}(g_1,[\tilde{q}_1]),
\] 
that is to perform the registration directly in the space $G\times\mathcal{Q}/G$, leading to the reduced hamiltonian flow in $T^*G\oplus T^*(\mathcal{Q}/G)$. This quotient shows that the decoupling yields a more consistent parameterization of vector fields by covectors, ensuring a clear separation of each deformation mode in the matching process and preventing operlapping effect. 

However, the asumption of a data attachment term $\mathcal{D}$ being $G$-invariant is a bit strong in practice. A first idea would be to replace the term $\mathcal{D}$ by the $G$-invariant term
\[
\Bar{\mathcal{{D}}}(\tilde{q}) = \inf_{g\in G} \mathcal{D}(g_1,g\cdot\tilde{q})
\]
This term could be computationally demanding, and we might want to keep $\mathcal{D}$ in the optimization. Nonetheless, by \cite{ARGUILLERE2015139}, minimizing $J$ is equivalent  to minimizing
    \[
    \tilde{J}_2(p^g_0,p_0) = \tilde{H}(e_G,q_S,p^g_0,p_0) + \mathcal{D}\left(g_1,\tilde{q_1}\right)
    \] where we shoot from $(e_G,q_S)$ by the initial covectors $(p^g_0,p_0)$. Moreover, since $V$ is $G$-invariant, we recall that the Hamiltonian can be rewritten as a sum of Hamiltonian functions on $T^*G$ and $T^*\mathcal{Q}$ 
    \[
    \tilde{H}(g,\tilde{q},p^g,p) = H^G(g,p^g) + H^V(\tilde{q},p).
    \] where $H^G(g,p^g)=\frac{1}{2} \vert K_V\xi_{\tilde{q}}^* p\vert_V^2$ and $H^V(\tilde{q},p)=\frac{1}{2} \vert K_{\mathfrak{g}}d_{e_G}R_g^*p^g\vert^2_{\mathfrak{g}}$. Therefore, in this setting, we can still restrict to initial covectors $p_0$ in $\mu^{-1}(0)$ in order to get hamiltonian flow that descends to the reduced space.

..\section{Application to rigid and diffeomorphic motions}\label{Sec:Rigid+diffeo}
In this section, we focus on the example of the rigid motions represented by the group of isometries and its action on space of curves. 

\subsection{The semidirect product of isometries and diffeomorphisms}

The group of isometries of $\R^d$ is a finite-dimensional Lie group particularly interesting in the computational anatomy framework since they represent classic and simple deformations. We denote $\Isom(\R^d) = \operatorname{SO}_d \ltimes \R^d$ the group of isometries and its associated Lie algebra $\mathfrak{isom}_d = \operatorname{Skew}_d \oplus \R^d$. First, we can notice that the group of isometries acts on $\R^d$ smoothly via diffeomorphisms and properly by $(R,T) \cdot x = Rx +T$. Moreover, the action of the group $\Isom(\R^d)$ on $\Diff_{C_0^k}(\R^d)$ defined by
\begin{equation}
        \rho_{(R,T)}(\varphi)(x)= R^{\top}\varphi(Rx+T)-R^{\top}T , \ x \in \R^d
\end{equation}
satisfies the assumptions of the framework presented in the subsection \ref{first_def}.
Then, we can consider the semidirect product $\operatorname{Isom}(\R^d) \ltimes \Diff_{C_0^k}(\R^d)$, which is a Banach half-Lie group, with group operations

\begin{equation}
  \left\{
    \begin{array}{l}
        (R',T',\varphi')(R,T,\varphi) = (R'R,R'T+T',R^{\top}\varphi'(R \varphi  +T)-R^{\top}T) \\
        (R,T,\varphi)^{-1}=(R^{\top},-R^{\top} T,R \varphi^{-1}(R^{-1}-R^{\top}T)+T)
         \end{array} 
    \right.
    \end{equation}

Given $(A,\tau,u) \in L^2([0,1],\mathfrak{isom}_d \oplus C_0^k(\R^d,\R^d))$, a flow in   $\operatorname{Isom}(\R^d) \ltimes \Diff_{C_0^k}(\R^d)$ is defined by :

\begin{equation}
    \left\{
        \begin{array}{ll}
            \dot{R}_t = A_tR_t \\
            \dot{T}_t = \tau_t + A_tT_t\\
            \dot{\varphi}_t=R_T^{\top} u_t(R_t\varphi_t + T_t) \\
            (R_0,T_0,\varphi_0)=(I_d,0,\id)
        \end{array} 
        \right.\notag
\end{equation}

\subsection{The shape space of curves}

For $k\geq 1$, let $D$ be either the interval $[0,1]$ for open curves or the unit circle $\mathbb{S}^1$ for closed curves. The space of  $C^k$-immersions in $\mathbb{R}^d$,

\begin{equation*}
    \operatorname{Imm}_{C^k}(D,\mathbb{R}^d) = \{ q \in C^{k}(D,\mathbb{R}^d) \mid q'(t)  \ne 0 \text{ for all } t \in D\}
\end{equation*}
is an open set of the Banach space $C^k(D,\mathbb{R}^d)$ which induces a manifold structure with tangent space $T_{q} \operatorname{Imm}_{C^k}(D,\mathbb{R}^d) \simeq C^k(D,\mathbb{R}^d)$, the space of $C^k$ vector fields along $q$. Reparameterization of curves arises from the action of the group of orientation-preserving diffeomorphism $\operatorname{Diff}^{+}(D)$ on the space of immersions :
\begin{equation*}
    \app{ \operatorname{Diff}^{+}(D) \times \operatorname{Imm}_{C^k}(D,\R^d) }{\operatorname{Imm}_{C^k}(D,\R^d)}{(\gamma,q)}{q\circ \gamma}
\end{equation*}
The space of curves that we would originally consider is the space of immersions modulo reparameterizations $\mathcal{S}(D,\R^d) = \operatorname{Imm}_{C^k}(D, \R^d) / \Diff^+(D)$. However since we will use a varifold data attachment term, which is invariant by reparameterization, we can identify the space of curves as the space of $C^k$-immersions . The group of diffeomorphism $\Diff_{C_0^k}(\R^d)$ acts on $\operatorname{Imm}_{C^k}(D,\R^d) $ by left composition $(\varphi,q) \mapsto \varphi \circ q $ with infinitesimal action $\xi : (X,q) \mapsto X \circ q$ which makes $\operatorname{Imm}_{C^k}(D,\R^d)$ a shape space in the sense of \cite{general_setting}. Moreover, the group of isometries acts on curves via 
\begin{equation*}
    (R,T) \cdot q  = Rq +T
\end{equation*}
which allows to define a shape space structure on curves with respect to the semidirect product of isometries by diffeomorphisms as stated in the following proposition. 
\begin{prop}[Shape space $\operatorname{Isom}(\R^d) \times \operatorname{Imm}_{C^k}(D,\R^d)$]
    The action of $\operatorname{Isom}(\R^d)$  on $\operatorname{Imm}_{C^k}(D,\R^d)$ 
    is a smooth action via diffeomorphisms and satisfies the compatibility condition \label{Ex:comp_hyp}
    \begin{equation*}
        \varphi \cdot ((R,T)\cdot q) = (R,T) \cdot (\rho_{(R,T)}(\varphi) \cdot q)
    \end{equation*}
Consequently, $\operatorname{Isom}(\R^d) \ltimes \operatorname{Diff}_{C_0^k}(\R^d)$ induces a shape space structure on $\operatorname{Isom}(\R^d) \times \operatorname{Imm}_{C^k}(D,\R^d)$.
\end{prop}

\begin{proof}
    The smoothness and the inversibility of the action can be deduced directly from its expression. 
    Using the definitions of the differents actions, the left and right hand side of the compatibility condition are equal. Indeed, on one hand we have 
    \begin{equation*}
        \varphi \cdot ((R,T)\cdot q)  = \varphi(Rq + T)
    \end{equation*}
and on the other hand,
    \begin{eqnarray*}
        (R,T) \cdot (\rho_{(R,T)}(\varphi) \cdot q) &=& (R,T)\cdot (R^{\top} \varphi(Rq+T)-R^{\top}T)\\
        &=& R(R^{\top} \varphi(Rq+T)-R^{\top}T) + T \\
        &=& \varphi(Rq+T).
    \end{eqnarray*}
    Finally, it follows  from proposition \ref{tilde_q_shape_space} that $\operatorname{Isom}(\R^d) \times \operatorname{Imm}_{C^k}(D,\R^d)$  is a shape space.
\end{proof}
In particular, the action of $\operatorname{Isom}(\R^d) \ltimes \operatorname{Diff}_{C_0^k}(D,\R^d)$ on $\operatorname{Isom}(\R^d) \times \operatorname{Imm}_{C^k}(D,\R^d)$ is defined by 
\begin{equation*}
 (R,T,\varphi) \cdot (R',T',q)  = (RR', RT'+T, R(\varphi \circ q) + T)
\end{equation*}
and leads to the following infinitesimal action, for $(A,\tau,u) \in \mathfrak{isom}_d \oplus  C_0^k(\R^d,\R^d)$:
\begin{equation*}
        (A,\tau,u) \cdot (R,T,q)  = (AR, AT+\tau, Aq + \tau + u\circ q)
\end{equation*}

\subsection{Matching problem and reduction}
Let $V \hookrightarrow C_0^{k+2}(\R^d,\R^d)$ be a RKHS induced by the gaussian kernel $k_{\sigma}(x,y)=\exp(-\frac{\Vert y -x\Vert^2}{2\sigma^2})$, i.e the Hilbert space defined by the completion of the vector space spanned by the functions $k_{\sigma}(x,\cdot) : y \mapsto k_{\sigma}(x,y)$ \cite{glaunes2005transport}. Let also $q_S\in \operatorname{Imm}_{C^k}(D,\R^d)$ be a source curve that we want to match to a target curve $q_T$. We consider the following minimization problem  

\begin{eqnarray}
    \label{energy_isom}
    \inf_{(A,\tau,v)} J(A,\tau,v) &=& \int_0^1 \frac{1}{2} \vert A_t \vert^2 + \frac{1}{2} \vert \tau_t \vert^2 + \frac{1}{2} \vert v_t \vert^2 \, dt 
+ \mathcal{D}((R_1,T_1), R_1 \tilde{q}_1 + T_1) \\
     \text{ s.t }& &     \left\{
        \begin{array}{l}
            \dot{R}_t = A_t R_t \\
            \dot{T}_t = A_t T_t + \tau_t \\
            \dot{\tilde{q}}_t = R_t^{-1} v_t(R_t \tilde{q}_t + T_t) \\
            (R_0,T_0,\tilde{q}_0) = (I_d,0,q_S)        
        \end{array} 
        \right.\notag
\end{eqnarray}
where  $\tilde{q} = R^{-1}(q-T) $. This allows to perform a registration between the source curve $q_S$ and the target, taking into account both rigid and non rigid deformations.

\begin{prop}[Hamiltonian associated with \eqref{energy_isom}]

The hamiltonian associated with the variational problem \eqref{energy_isom} is 

\begin{equation}
H(R,T,\tilde{q},\tilde{p}^A,\tilde{p}^{\tau},\tilde{p}) = \frac{1}{2} (\vert \tilde{p}^A R^{\top} + \tilde{p}^{\tau} T^{\top}\vert^2  + \vert \tilde{p}^{\tau} \vert^2 + \vert  K_V \xi_{\tilde{q}}^* \tilde{p}\vert^2)
\end{equation}
 and the critical points of the energy satisfy the geodesic equations.
\end{prop}

\begin{proof}
    
The Hamiltonian associated with this problem is 

\begin{equation*}
    H(R,T,\tilde{q},\tilde{p}^A,\tilde{p}^{\tau},\tilde{p},A,\tau,v) = (\tilde{p}^A \vert AR) + (\tilde{p}^{\tau} \vert AT +\tau) + (\tilde{p} \vert R^{-1}v(R\tilde{q}+T)) - \frac{1}{2}(\vert A \vert^2 + \vert \tau\vert^2 +\vert v \vert^2)
\end{equation*}
The geodesic equations lead to the following Hamiltonian 
\begin{equation*}
    H(R,T,\tilde{q},\tilde{p}^A,\tilde{p}^{\tau},\tilde{p}) = \frac{1}{2} \vert \tilde{p}^A R^{\top} + \tilde{p}^{\tau} T^{\top}\vert^2  + \frac{1}{2}\vert \tilde{p}^{\tau} \vert^2 + \vert  K_V \xi_{R\tilde{q}+T}^* (R\tilde{p}) \vert^2
\end{equation*}
The invariance by isometries of the kernel implies that $ \vert  K_V \xi_{R\tilde{q}+T}^* (R\tilde{p}) \vert = \vert K_V \xi_{\tilde{q}}^* \tilde{p} \vert$.
Therefore, the critical points of the energy satisfy the geodesic equations following proposition \ref{critical_point_H}.
\end{proof}
Thanks to this invariance, we can perform reduction as presented in the previous section. The action of isometries on immersions $(R,T) \cdot q = Rq +T$ is proper but is not necessarily free. Indeed, certain symmetric immersions, such as straight lines or circles, are stabilized by nontrivial subgroups of $\operatorname{SO}(d)$. To overcome this issue, we consider the space of immersions that have trivial stabilizer.
\begin{equation}
    \operatorname{Imm}_{C^k}(D,\R^d)_e = \{ q \in \operatorname{Imm}_{C^k}(D,\R^d) \mid \operatorname{Stab}(q)= \{e\}\}
\end{equation} This assumption is not too restrictive since this space is an open set of the space of immersions, so it is a manifold too, and is dense in $\operatorname{Imm}_{C^k}(D,\R^d)$, so it still represents almost all the curves we will be interested in.  Moreover, the invariance of the kernel implies that the space $V$ is $\operatorname{Isom}(\R^d)$-invariant under the action of $\operatorname{Isom}(\R^d)$, i.e  $d_{\id} \rho_{(R,T)} : v \mapsto (x \mapsto R^{-1}v(Rx +T))$ is an isometry and preserves $V$. Then, we define the momentum map $\mu : T^*\operatorname{Imm(D,\R^d)}_e \rightarrow \mathfrak{isom}_d^*$ by
\begin{eqnarray*}
    (\mu(q,p) \vert (A,\tau) ) &=& (\xi^{\mathfrak{g}*}_qp \vert (A,\tau)) \\
    &=& \frac{1}{2} \int_D  \operatorname{Tr}((q(s) p(s)^{\top} - p(s)q(s)^{\top})A) ds + \int_D p(s)^{\top} \tau \, ds \\
    &=& (\mu_A(q,p)\vert A) + (\mu_{\tau}(q,p) \vert \tau)
\end{eqnarray*}
The momentum map $\mu_A$ corresponds to the angular momenta in classical mechanic, denoted $p \wedge q$.
Moreover, Noether's theorem states that the momentum maps $\mu_A(q,p)$ and $\mu_{\tau}(q,p)$ are constants of the motion, which we can retrieve by direct computations.
\begin{eqnarray*}
    \frac{d}{dt} \mu_{\tau}(q_t,p_t) &=& \int_D \dot{p}_t(s) ds \\
    &=& - \int_D \int_D \nabla k_{\sigma}(q(s),q(s')) \langle p(s),p(s')\rangle  + \int_D Ap(s)^{\top}\\
    &=& 0
\end{eqnarray*}
and 
\begin{eqnarray*}
    \frac{d}{dt}\mu_{A}(q_t,p_t) &=& \int_D \dot{q}(s)p(s)^{\top} + q(s)\dot{p}(s)^{\top} - (\dot{p}(s)q(s)^{\top} + p(s)\dot{q}(s)^{\top})) ds \\
    &=& \int_D \int_D k(q(s),q(s'))p(s')p(s)^{\top} - q(s) \nabla k(q(s),q(s'))^{\top} \langle p(s),p(s')\rangle ds ds' \\
    &+& \int_D \int_D \nabla k(q(s),q(s'))q(s)^{\top} \langle p(s),p(s')\rangle- k(q(s),q(s'))p(s)p(s')^{\top} ds ds' \\
    &=& 0
\end{eqnarray*}
So, if we enforce that $(\mu_A,\mu_{\tau})=0$ at $t=0$, then $(\mu_A,\mu_{\tau})=0$ for all time t.
Thus, by restricting to the zero level set of the momentum map and considering the canonical projection $\pi : \mathfrak{isom}_d^* \times (\mu_A,\mu_{\tau})^{-1}(0) \rightarrow \mathfrak{isom}_d^* \times (\mu_A,\mu_{\tau})^{-1}(0)/\operatorname{Isom}(\R^d)$.  
, we can define a reduced Hamiltonian on 
\begin{equation}
    \mathfrak{isom}_d^* \times ( (\mu_A,\mu_{\tau})^{-1}(0)/\operatorname{Isom}(\R^d ) )\simeq \mathfrak{isom}_d^* \times (T^*\operatorname{Imm}_e(D,\R^d)/\operatorname{Isom}(\R^d))
\end{equation} 
Coming back to an inexact match problem, we want to minimize the functional $J$ defined in \ref{energy_isom} over $(\mu_A,\mu_{\tau})^{-1}(0)$ using a two-term data attachment:
$$
J(p^A_0,p^\tau_0,p_0) = H(p^A_0,p^\tau_0,p_0) + \mathcal{D}_{\mbox{rigid}}(R_1,T_1) + \mathcal{D}(R_1\tilde{q}_1+T_1)
$$
So, enforcing the constraints $(\mu_A,\mu_{\tau})=(0,0)$ leads to minimizing $J$ over $ \mathfrak{isom}_d^* \times (T^*\operatorname{Imm}_e(D,\R^d)/\operatorname{Isom}(\R^d))$ instead of $\mathfrak{isom}_d^* \oplus T^{*}\operatorname{Imm}_e(D,\R^d)$. A consequence is that the rigid and diffeomorphic deformations are decoupled and have clear separate contributions.
 Note that if we choose a coefficient in front of the rigid energy small enough, this problem becomes approximately equivalent to replacing $\mathcal{D}(R_1q_1+T_1)$ by $\Bar{\mathcal{D}}(\tilde{q}_1)=\inf_{R,T}\mathcal{D}(Rq +T)$.
\begin{rem}
    Since $(p^A_0,p^\tau_0,p_0) \in (\mu_A,\mu_{\tau})^{-1}(0)$, then we simply have
    \begin{equation*}    H(p_0^A,p_0^\tau,p_0)=\frac{1}{2}\left(\lvert p_0^A \rvert^2 + \lvert p_0^\tau\rvert^2 + \lvert K_V \xi_{q_0}^*p_0\rvert^2 \right)
    \end{equation*}
\end{rem}

\subsection{Numerical applications}

For numerical applications, curves in $\R^d$ are encoded by a set of landmarks $(q_i)_{i=1,...,n} \in (\R^d)^n$ linked by edges belonging to $F_q \subset \llbracket 1,n \rrbracket^2 $. Then, the action of vector fields on discrete curve is defined by the action on landmarks $v \cdot q_i=v(q_i)$. The covectors associated to the curve $q$ are represented by vectors $p_i \in \R^d$ with base point $q_i$. Then, the hamiltonian can be expressed as a discrete sum.
\begin{equation}
    H(R,T,\tilde{q},\tilde{p}^A,\tilde{p}^{\tau},\tilde{p}) = \frac{1}{2} \left(\vert \tilde{p}^A R^{\top} + \tilde{p}^{\tau} T^{\top} \vert^2 + \vert \tilde{p}^{\tau} \vert^2 + \sum_{i,j}\tilde{p}_i^{\top} K(\tilde{q}_i,\tilde{q}_j)\tilde{p}_j \right)
\end{equation}
In this framework, the constraints can be expressed as a sum over the landmarks 
\begin{equation} \label{constraints_discrete_curve}
\left\{
\begin{aligned}
    \mu_A = 0 &\;\Longleftrightarrow\; \sum_i \left( \tilde{p}_i \tilde{q}_i^\top - \tilde{q}_i \tilde{p}_i^\top \right) = 0 \\
    \mu_\tau = 0 &\iff \sum_i \tilde{p}_i = 0
\end{aligned}
\right.
\end{equation}
The covector $p$ have to be orthogonal to the curve to preserve the invariance by reparameterization, which can be expressed by projecting the covectors on the orthogonal of the curve $\tilde{p} \mapsto \tilde{p} - \langle \tilde{p},\tilde{q}\rangle \tilde{q}$. Therefore, the conditions on covectors can be enforced by projecting them onto the null space 
\begin{equation}
    \Big\{(\tilde{p}_i)_{i=1}^n\in (\mathbb R^d)^n\ \Big\vert\ \sum_{i=1}^n(\tilde{p}_i\tilde{q}_i^\top-\tilde{q}_i\tilde{p}_i^\top)=0,\ \sum_{i=1}^n \tilde{p}_i=0,\ \text{and }\langle \tilde{q}_i,\tilde{p}_i\rangle=0 \ \text{for all }i\in [\![1,n]\!]\Big\}
\end{equation}
However, we will often not impose the conditions $\langle \tilde{q}_i, \tilde{p}_i\rangle=0$ required for invariance by reparameterization, but will instead use a varifold data attachment that implicitly incorporates them \cite{Charon_2013}.

\begin{rem}
    If we want to enforce the condition on rotations only, $\mu_A= 0$ can be satisfied by projecting the covectors $\tilde{p}_i$ onto the null space of the functional $F : \tilde{p} \mapsto \sum_i \tilde{p}_i\tilde{q}_i^{\top} - \tilde{q}_i \tilde{p}_i^{\top}$ using a Singular Value Decomposition (SVD). On the other hand, to enforce the condition on translation only $\mu_{\tau}=0$, we can center the covectors $\tilde{p}- \sum_i \tilde{p}_i$. Due to Noether's theorem, these conditions have to be satisfied at time t=0 only and then will remain equal to zero during the motion. 
\end{rem}

\paragraph{Results}  
In this section, we present numerical experiments illustrating the matching problem and the influence of constraints $ \mu_A = 0$. We consider a registration through a deformation generated by a diffeomorphic deformation and a rotation. The optimization problem is thus
\begin{eqnarray}
\label{var_pb_numexp}
    \inf_{(A,v) \in L^2([0,1],\mathfrak{so}_d \times V)} J(A,v) &=& \int_0^1 \frac{1}{2} \vert A_t \vert^2 + \frac{1}{2} \vert v_t \vert^2 \, dt + \mathcal{D}(R_1 \tilde{q}_1) \\
     \text{ s.t }& &     \left\{
        \begin{array}{l}
        \dot{R}_t = A_t R_t \\
        \dot{\tilde{q}}_{i,t} = R_t^{-1} v_t(R_t \tilde{q}_{i,t}) \quad \text{ for } i \in \{ 1,\ldots,n \} \\
        (R_0,\tilde{q}_0) = (I_d,q_S)        
        \end{array} 
        \right.\notag
\end{eqnarray}
where  $\tilde{q} = R^{-1}q $.  We consider a varifold data attachment term $\mathcal{D}$ which is particularly well-adapted to the matching of curves due to the invariance by reparameterization \cite{Charon_2013}. Ideally, the rotation component of the model would encode the full rotation between the source and the target, while the diffeomorphic deformation would capture remaining deformations. We represent $\tilde{q}=R^{-1}q$, which isolates the action of the diffeomorphic deformation by removing the rotational component. Consequently, we expect that the geodesic $\tilde{q}_t$ does not rotate.

\begin{figure}[!ht]
\centering
\begin{subfigure}{0.485\textwidth}
        \centering
        \includegraphics[width=0.4\linewidth]{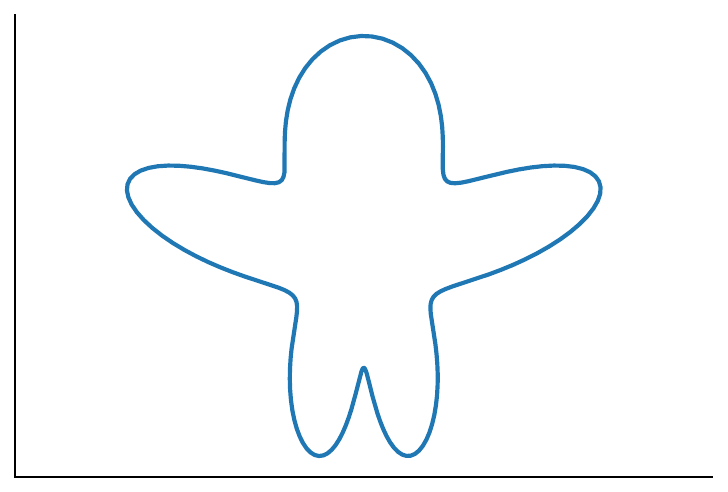}
        \caption{Source}
        \label{source1}
\end{subfigure}
\hspace{-3cm}
\begin{subfigure}{0.485\textwidth}
        \centering
        \includegraphics[width=0.4\linewidth]{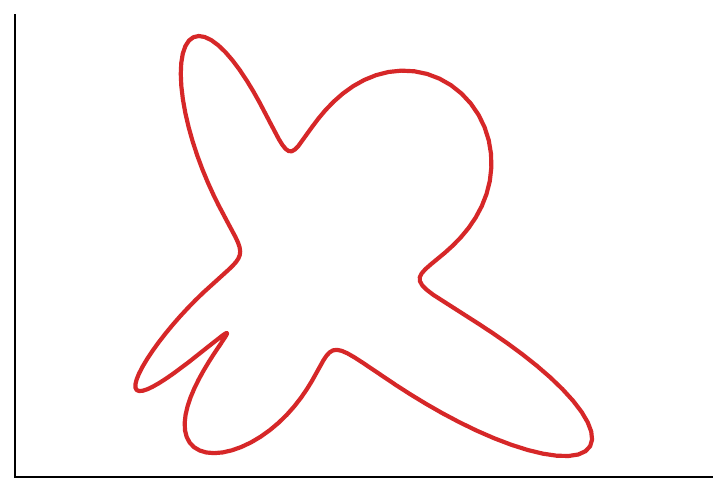}
        \caption{Target}
        \label{target1}
\end{subfigure}
\vspace{1em}

\begin{subfigure}{0.9\textwidth}
    \centering  
\begin{minipage}[c]{0.8\textwidth}
\centering
\includegraphics[scale=0.7]{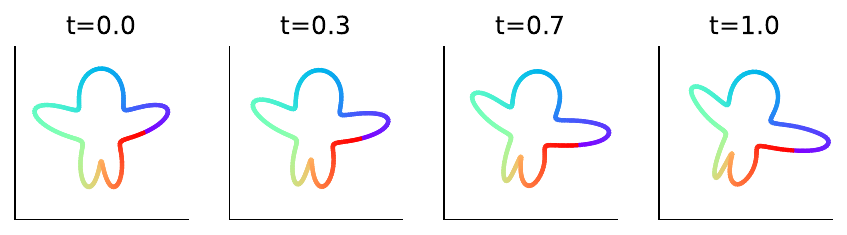}
\end{minipage}%
\begin{minipage}[c]{0.2\textwidth}
    \centering
    \includegraphics[scale=0.3]{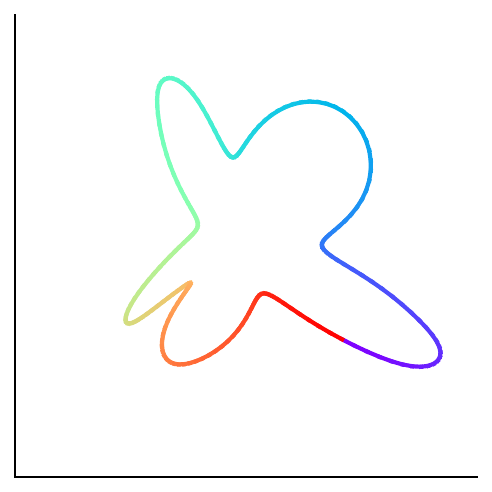}
\end{minipage} 
\caption{Without constraints $\mu_A=0$}
\label{fig:no_constraints1}
\end{subfigure}

 \vspace{0.5em}
\begin{subfigure}{0.9\textwidth}
\centering
\begin{minipage}[c]{0.8\textwidth}
\centering
\includegraphics[scale=0.7]{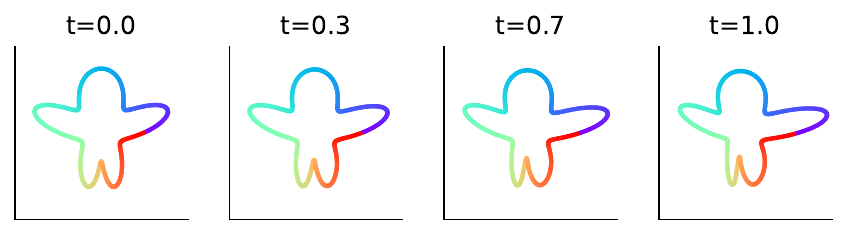}
\end{minipage}%
\begin{minipage}[c]{0.2\textwidth}
    \centering
    \includegraphics[scale=0.3]{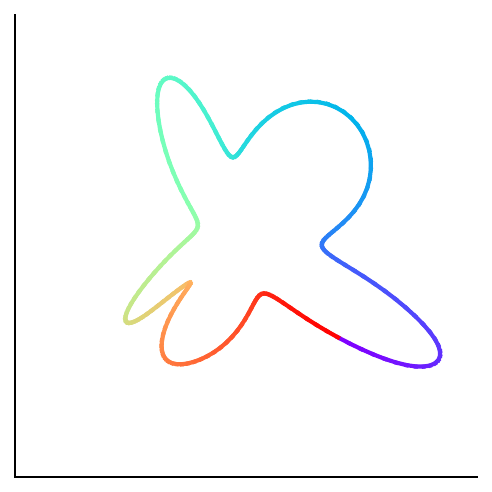}
\end{minipage}
\caption{With constraints $\mu_A=0$}
\label{fig:constraints}
\end{subfigure}

\caption{\textbf{Influence of the constraints on geodesics}. This experiment illustrates a matching task between the source \ref{source1} and the target \ref{target1}. The matching is performed through deformations generated by rotations and diffeomorphic deformation induced by a Gaussian RKHS, solving problem \eqref{var_pb_numexp}. For both experiments, we show the evolution of $\tilde{q}_t = R_t^{-1} q_t$, in order to represent only deformations induced by diffeomorphisms. Figures \ref{fig:no_constraints1} and \ref{fig:constraints} illustrate the registration performed without and with the constraint $\mu_A = 0$ enforced at time $t=0$, respectively.  The right column corresponds to the final matching where both the rotation and the  diffeomorphism are applied. }
\label{Mult:fig:toy_example_curve1}
\end{figure}

Figure \ref{Mult:fig:toy_example_curve1} compares the effect of enforcing the constraint $\mu_A = 0$ \eqref{constraints_discrete_curve} at time $t=0$ when matching the source \ref{source1} onto the target \ref{target1}.
Figure \ref{fig:no_constraints1} shows that without constraints, the diffeomorphic deformation can capture rotational motion while this issue does not arise when the constraints are enforced. We highlight that the constraints reduce the space of possible non-rigid deformations, and thereby facilitate the convergence of the LDDMM component.

\begin{figure}[h!]
    \centering

    \begin{subfigure}{0.485\textwidth}
        \centering
        \includegraphics[width=0.4\linewidth]{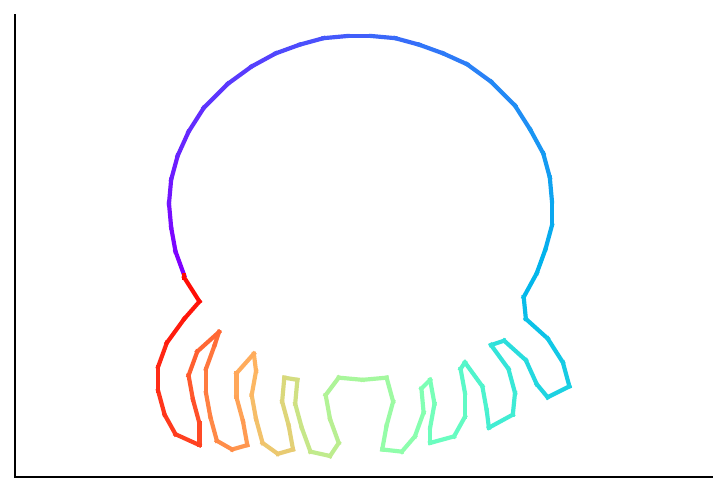}
        \caption{Source}
        \label{source2}
    \end{subfigure}
    \hspace{-3cm}
    \begin{subfigure}{0.485\textwidth}
        \centering
        \includegraphics[width=0.4\linewidth]{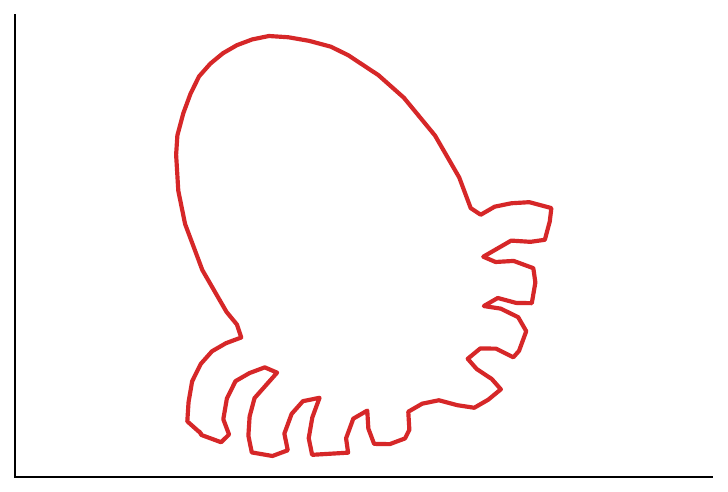}
        \caption{Target}
        \label{target2}
    \end{subfigure}

    \vspace{0.5em}
    \begin{subfigure}{0.8\textwidth}
        \centering
        \includegraphics[width=\linewidth]{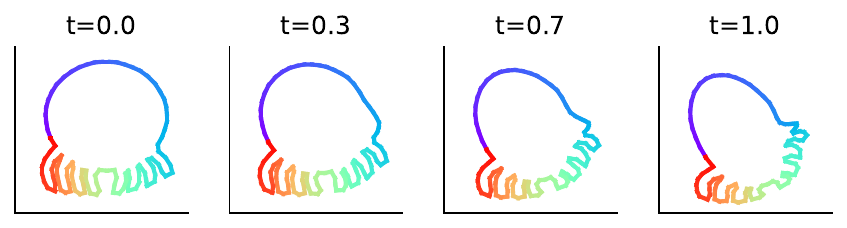}\\ 
           \caption{Rigid alignment followed by LDDMM}
           \label{rigid_lddmm_1}
    \end{subfigure}

    \vspace{0.5em}
    \begin{subfigure}{0.8\textwidth}
        \centering
     \includegraphics[width=\linewidth]{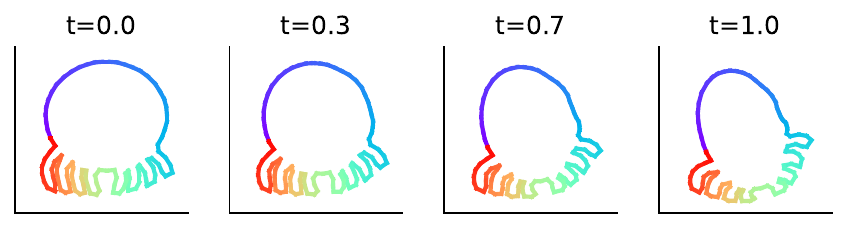}\\
     \caption{Joint optimization without constraints $\mu_A=0$}
     \label{without_constraints_1}
    \end{subfigure}    \vspace{0.5em}
    \begin{subfigure}{0.8\textwidth}
        \centering
        \includegraphics[width=\linewidth]{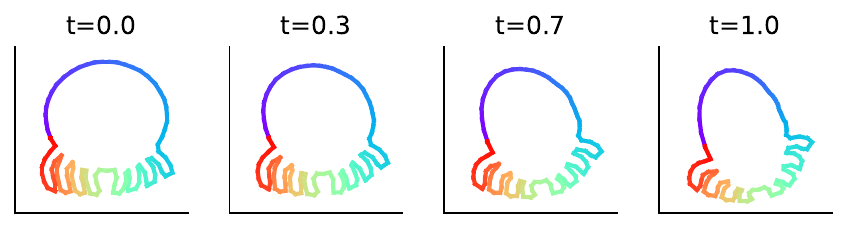}
           \caption{Joint optimization with constraint $\mu_A = 0$}
        \label{with_constraints_1}
    \end{subfigure}
    \caption{\textbf{Comparison of registration methods.} This experiment illustrates a matching task between the source \ref{source2} and the target \ref{target2}, minimizing three different variational problems using deformations generated by rotations and diffeomorphic deformations induced by a Gaussian RKHS. Second row \ref{rigid_lddmm_1} corresponds to a diffeomorphic registration after a first rigid pre-alignment. The two bottom rows \ref{without_constraints_1} and \ref{with_constraints_1} correspond to minimization of problem \eqref{var_pb_numexp} where we respectively not enforced and then enforced the constraint $\mu_A=0$. For all the experiments, we show the evolution of $\tilde{q}_t=R_t^{-1}q_t$ in order to represent only deformations induced by diffeomorphisms.}
    \label{Mult:fig:toy_example_curve3}
\end{figure}

\begin{figure}[h!]
    \centering

    \begin{subfigure}{0.485\textwidth}
        \centering
        \includegraphics[width=0.4\linewidth]{figures/octopus_goat2/source1.pdf}
        \caption{Source}
    \end{subfigure}
    \hspace{-3cm}
    \begin{subfigure}{0.485\textwidth}
        \centering
        \includegraphics[width=0.4\linewidth]{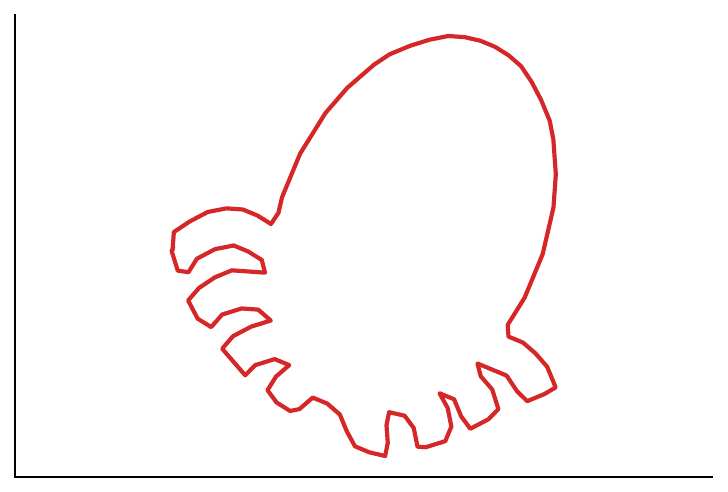}
        \caption{Target}
    \end{subfigure}

    \vspace{0.5em}
    \begin{subfigure}{0.8\textwidth}
        \centering
        \includegraphics[width=\linewidth]{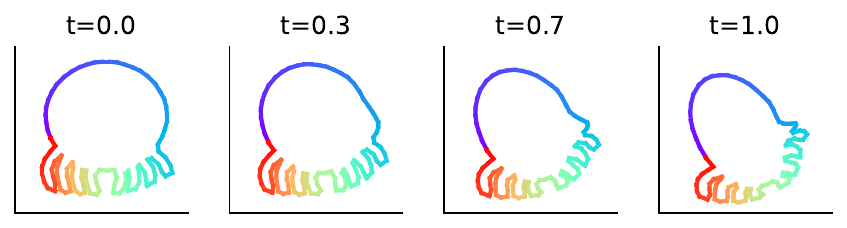}\\ 
           \caption{Rigid alignment followed by LDDMM}
        \label{rigid_lddmm_2}
    \end{subfigure}

    \vspace{0.5em}
    \begin{subfigure}{0.8\textwidth}
        \centering
     \includegraphics[width=\linewidth]{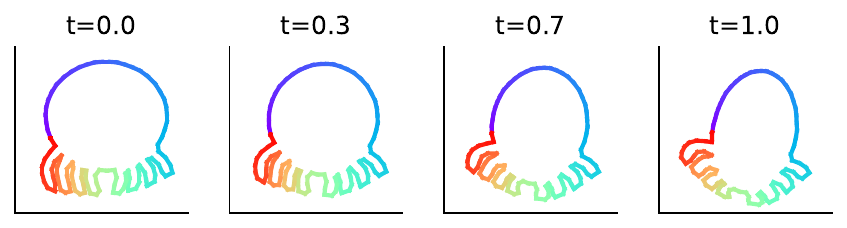}\\
     \caption{Joint optimization without constraint $\mu_A=0$}
    \end{subfigure}    \vspace{0.5em}
    \begin{subfigure}{0.8\textwidth}
        \centering
        \includegraphics[width=\linewidth]{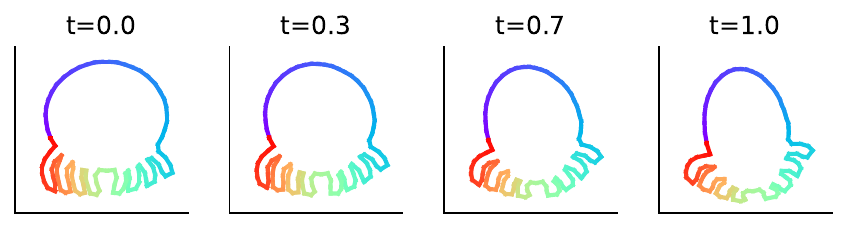}
           \caption{Joint optimization with constraint $\mu_A=0$}
    \end{subfigure}
    \caption{\textbf{Comparison of registration methods.} Similarly to Figure \ref{Mult:fig:toy_example_curve3}, this experiment illustrates the same matching tasks, with the same source, while the target has been slightly rotated.}
    \label{Mult:fig:toy_example_curve4}
\end{figure}

We further experimented our method on the dataset \cite{Enseiht}, shown in Figures \ref{Mult:fig:toy_example_curve3} and in a similar experiment with the target rotated \ref{Mult:fig:toy_example_curve4}. For these experiments, we also compare the joint optimization (problem \eqref{var_pb_numexp}) with the standard two-stage procedure, where rigid pre-alignment precedes diffeomorphic registration, and analyze the resulting differences. Figures illustrate that, when solving problem \eqref{var_pb_numexp} with the constraints enforced, the evolution of $\tilde{q}_t$ does not depend on the initial orientation of either the target or the source. Moreover, in all these experiments, our method consistently yields the most accurate matching, while the two-step approach and the joint optimization without the constraint enforced produce diffeomorphic deformations that perform a rotation too.  In particular, the two-step optimization presented in Figure \ref{rigid_lddmm_2} demonstrates that the diffeomorphic deformation rotates the shape in the opposite direction of the target. This results from the pre-alignement step over-rotating the shape, which the diffeormorphism subsequently attempts to compensate. Furthermore, as illustrated in Figures \ref{Mult:fig:toy_example_curve3} and \ref{Mult:fig:toy_example_curve4}, the orientation of the target causes an emergence of an additional limb in the two-step approach. This phenomenon arises because the rotation becomes fixed during the second optimization phase (purely diffeomorphic registration), leading to an overly constrained deformation process.

\newpage \newpage \newpage
\section{Transport of landmarks using anisotropic Gaussian kernels}
\label{Ani:Sec:sp_const_kernels}
In this section, we assume that the space $V$ generating vector fields for the diffeomorphism part is a RKHS induced by an anisotropic Gaussian kernel that allows to favor motions along certain privileged axes. Such a kernel is built by replacing the classic euclidian norm in the gaussian kernel by an anisotropic metric encoded by a symmetric definite positive matrix. While deforming shapes, we would like to keep track of the favorized axis encoded by this metric by transporting it along with the shape during the motion (cf. Figures \ref{Ani:fig:toy_example_scaling_fixed} and \ref{Ani:fig:toy_example_scaling_non_fixed}). To do so, we enrich the shape space by considering the metric as a part of it.  Following \ref{Sec:diff_structure} we define the group of deformations as a semidirect product of scalings, isometries and diffeomorphisms, acting on both the metric and the shape encoded as landmarks.

In this setting, we will not obtain a right-invariant sub-Riemannian metric on the group of deformations, since the Gaussian kernel depends now on the metric which is part of the shape. 

\subsection{Anisotropic Gaussian kernel}
Let $S_d^{++}$ be the space of symmetric definite positive matrices on $\R^d$. Any symmetric definite positive matrix $\Sigma\in S_d^{++}$ defines an anistropic scalar product $\langle x,y\rangle_\Sigma = \langle \Sigma x, y\rangle$ and its associated anisotropic norm $\lVert q\rVert_\Sigma^2=\langle q,q\rangle_\Sigma$. We define the (normalized) anisotropic Gaussian kernel associated with this metric :
\[
k_\Sigma(x,y) = \exp\left(-\frac{1}{2}\lVert x-y\rVert_{\Sigma^{-1}}^2\right)\Sigma.
\]
and we denote $V_\Sigma$ the reproducing kernel Hilbert space associated to $k_\Sigma$ and $K_\Sigma$ : $V_\Sigma^* \to V_\Sigma$ the Riesz isomorphism.

Note that in particular, if $\Sigma$ is not a multiple of the identity, i.e. if its eigenvalues are distinct, then the kernel $k_\Sigma$ is not invariant by rotation anymore, compared to classic isotropic Gaussian kernels. This highlights the fact that the metric $\Sigma$ privileges some axis.

\subsubsection*{A first (naive) model}

A first naive idea would be to use this Gaussian kernel to perform a coupled deformation of isometries and diffeomorphisms as in section \ref{Sec:diff_structure}. However, the introduction of anisotropy in the space $V_\Sigma$ through the metric $\Sigma$ will make the deformation induced by the diffeomorphism easier along certain directions, leading to a lower energy cost and tends to align the shape along those favorized axis. We illustrate this phenomenon with an example on the space of $n$-landmarks $\operatorname{Lmk}_n(\R^d)$.

We define, following section \ref{Sec:Rigid+diffeo}, the group $SO_d\ltimes \Diff_{C_0^k}(\R^d)$ of deformations, allowing to couple rotations and diffeomorphisms. This group acts on the shape space $\operatorname{Lmk}_n(\R^d)$ by 
\[
(R,\varphi)\cdot (q_i)_{i\leq n} = (R\varphi(q_i))_{i\leq n}
\]
and the infinitesimal action becomes $(A,u)\cdot (q_i)_{i\leq n} = (Aq_i +u(q_i))_{i\leq n}$ where $(A,u)\in \mathfrak{so}_d\oplus C_0^k(\R^d,\R^d)$. This allows, given a source shape $\bq_S = (q_{S,i})$ and a target $\bq_T =(q_{T,i})$, to introduce the matching problem (cf. general problem \ref{var_pb})
\begin{eqnarray}
    \label{Ani:var_pb_an_naive}
    \inf_{(A,u) \in L^2([0,1],\mathfrak{so}_d \times V_\Sigma)} J(A,u) &=& \frac{1}{2}\int_0^1 \vert A \vert_{\mathfrak{so}_d}^2 +  \vert v_t \vert_{V_\Sigma}^2 \, dt + \mathcal{D}(q_1) \\
     \text{ s.t }& &     \left\{
        \begin{array}{l}
            \dot{q}_{i,t} = A_t q_{i,t} + u_t(q_{i,t}) \\
            q_{i,0} = q_{S,i}        
        \end{array} 
        \right.\notag
\end{eqnarray}
where $\mathcal{D}:\operatorname{Lmk}_n\to\R$ is a varifold data attachment term \cite{Charon_2013,BME} defined by

\[
\mathcal{D}(\bq) = \sum_{i,j} k_W(q_i,q_j) + k_W(q_{T,i},q_{T,j}) - 2k_W(q_i,q_{T,j}) 
\]
and $k_W$ is a Gaussian kernel (or a sum of Gaussian kernels).
\paragraph{Results} We test this model by minimizing variational problem \ref{Ani:var_pb_an_naive} with a given source and target represented in Figure \ref{Ani:fig:source_target}. Here the source and target shape are encoded by landmarks with $200$ points, and represents elongated shapes along the $y$-axis with the target being larger than the source. We use for the diffeomorphisms an anisotropic gaussian kernel with scales$\Sigma_S=\begin{bmatrix}
1 & 0  \\
0        & 0.1
\end{bmatrix}$, meaning the deformation is easier along the $x$-axis and harder along the $y$-axis. We show in \ref{Ani:fig:toy_example_scaling_fixed} the geodesic obtained after the minimization of \eqref{Ani:var_pb_an_naive}. We clearly observe that the shape tends to align along the $x$-axis using the rotation part of the deformation, before performing the stretching of the shape necessary to obtain the target with a lower cost. Indeed, the minimization algorithm manages to use the rotation part in order to align the shape with the $x$-axis (favored by the anisotropic kernel) instead of performing a vertical stretching which is the expected deformation.  This problem is due to the non-transport of the anisotropy by the rotation. Indeed, the axis of the anisotropy are fixed so the algorithm will tend to first align with the favored axis with the rotation, deforms with the diffeomorphism and then rotating back to the unfavored axis. If the axis of the anisotropy are transported, this trick would not work anymore since they will rotate along with the shape. Note that the alignment with the $x$-axis is not complete and only partial here, this is because the cost for the rotation is not free.

\begin{figure}[ht]
\centering
\includegraphics[scale=0.5]{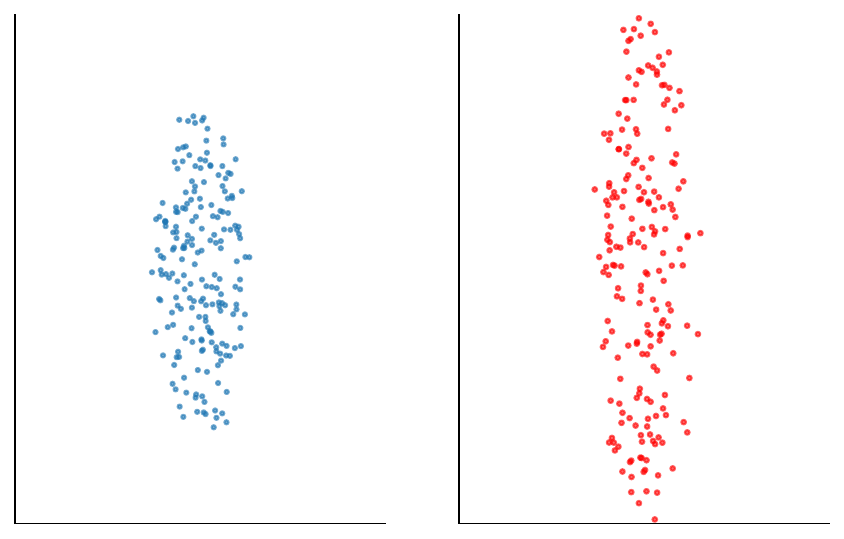}
\caption{\emph{Performing jointly rigid and non rigid registration, with an anisotropic kernel.} Source shape (in blue), $200$ points and target shape (in red), $200$ points.}
\label{Ani:fig:source_target}
\end{figure}

\begin{figure}[ht]
\centering
\includegraphics{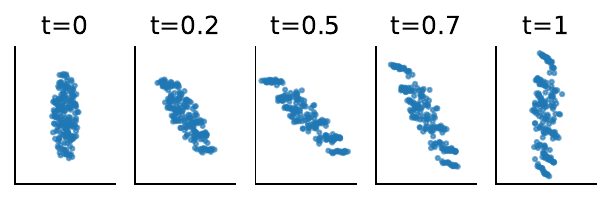}
\caption{\emph{Performing jointly rigid and non rigid registration, minimizing problem \ref{Ani:var_pb_an_naive} with an anisotropic kernel.} This experiment illustrates a matching task between the source (blue) and the target (red). The matching is performed through deformations generated by rotations and diffeomorphic deformation induced by a anisotropic gaussian kernel with scale $(\sigma_x,\sigma_y)=(1.,0.1)$.}
\label{Ani:fig:toy_example_scaling_fixed}
\end{figure}

\subsection{Group of deformations}
In this section, we define the group of deformations as a semidirect product between the finite dimensional group composed of scalings and isometries, and the group of diffeomorphisms, following section \ref{Sec:diff_structure}. As a first step to tackle the issue illustrated in Figure \ref{Ani:fig:source_target}, we will define an action of this semidirect product on the space of metrics $S_d^{++}$ such that the anisotropic metric is also transported by the group of deformations, in particular through the action of isometries and scalings.

\subsubsection{The semidirect product of isometries and scalings}
We define the (finite dimensional) Lie group $\alpha\operatorname{-Isom}(\R^d)=(\R_{>0} \times \operatorname{SO}_d)\ltimes\R^d$ consisting of scalings, rotations and translations, with the composition law
\[
(\alpha, R, T)\cdot(\alpha',R',T') = (\alpha \alpha', RR', \alpha RT' + T)
\]
The group $\alpha\operatorname{-Isom}(\R^d)$ acts on $\R^d$ via the transformation $\varphi_{(\alpha,R,T)}$ defined by
\[
\varphi_{(\alpha,R,T)} : \app{\R^d}{\R^d}{q}{\alpha Rq + T}
\]
so that \[\varphi_{(\alpha,R,T)}\circ\varphi_{(\alpha',R',T')}=\varphi_{(\alpha,R,T)\cdot(\alpha',R',T')}\]
 The group of rigid motions $\alpha\operatorname{-Isom}(\R^d)$ also acts on the left on the space $S_d^{++}$ in the following way 
\[
(\alpha,R,T)\cdot \Sigma= \alpha^2 R\Sigma R^\top
\]
This action is compatible with the definiton of the anisotropic Gaussian kernel we defined before in the sense that we get isometries between RKHS.

\begin{prop}[Isometries between Gaussian RKHS]
\label{Ani:isom_RKHS}
Let $\Sigma\in S_d^{++}$ and $u\in V_\Sigma$, where $V_\Sigma$ is the Gaussian RKHS induced by the metric $\Sigma$. Let $(\alpha,R,T)\in \alpha\operatorname{-Isom}(\R^d)$, and define the vector field $u'=d\varphi_{(\alpha,R,T)^{-1}}u\circ\varphi_{(\alpha,R,T)}$ by
\[
\forall q \in\R^d, \quad u'(q) = \alpha^{-1}R^\top u(\alpha Rq+T).
\]
Then $u'\in V_{\Sigma'}$ with $\Sigma'=(\alpha,R,T)^{-1}\cdot \Sigma=\alpha^{-2}R^\top \Sigma R$ and we get 
\[
\lvert u'\rvert_{V_{\Sigma'}} = \lvert u \rvert_{V_\Sigma}
\]
\end{prop}
\begin{rem}
    In particular, the map
    \[
    \app{V_\Sigma}{V_{\alpha^{-2}R^\top \Sigma R}}{u}{d\varphi_{(\alpha,R,T)^{-1}}u\circ\varphi_{(\alpha,R,T)}}
    \] is an isometry of Hilbert spaces.
\end{rem}
\begin{proof}
Since the family $\{k_\Sigma(q,\cdot)b\mid q,b\in\R^d\}$ forms a total subset of the RKHS $V_\Sigma$ \cite{Aro50}, we restrict to the case  $u = \sum_i k_\Sigma(q_i,\cdot)b_i$ for some $q_i,b_i\in\R^d$. Let $(\alpha,R,T)\in\alpha\operatorname{-Isom}(\R^d)$, we get for any $q\in\R^d$, 
\begin{align*}
    d&\varphi_{(\alpha,R,T)^{-1}}u\circ\varphi_{(\alpha,R,T)}(q)=  \sum_i\alpha^{-1}R^\top k_\Sigma(q_i,\alpha Rq + T) b_i\\ 
    &= \sum_i \exp\left(-\frac{1}{2}\langle \Sigma^{-1}(q_i - \alpha Rq - T),q_i -  \alpha Rq - T\rangle\right) \alpha^{-1}R^\top \Sigma b_i\\
    &= \sum_i\exp\left(-\frac{1}{2}\langle \alpha^2R^\top \Sigma^{-1}R(\alpha^{-1}R^\top( q_i - T) - q),\alpha^{-1}R^\top( q_i - T) - q \rangle\right) \alpha^{-2}R^\top \Sigma R (\alpha R^\top b_i)\\
    &= \sum_i k_{\alpha^{-2}R^\top \Sigma R}(\alpha^{-1}R^\top( q_i - T), q) \alpha Rb_i 
\end{align*}
and thus $d\varphi_{(\alpha,R,T)^{-1}}u\circ\varphi_{(\alpha,R,T)} = K_{\alpha^{-2}R^\top \Sigma R}\sum_i\delta_{\alpha^{-1}R^\top( q_i - T)}^{\alpha R^\top b_i} \in V_{(\alpha,R,T)^{-1}\cdot \Sigma}$. Moreover, by definition we see that 
\begin{align*}
\lvert&\varphi_{(\alpha,R,T)^{-1}}u\circ\varphi_{(\alpha,R,T)}\rvert_{V_{(\alpha,R,T)^{-1}\cdot \Sigma}}^2 =\sum_{i,j} (\alpha R^\top b_i)^\top k_{\alpha^{-2}R^\top \Sigma R}(\alpha R^\top( q_i - T), \alpha R^\top( q_j - T)) \alpha R^\top b_j \\
   &= \sum_{i,j}\exp\left(-\frac{1}{2}\langle \alpha^2R^\top \Sigma^{-1}R(\alpha^{-1}R^\top( q_i - q_j)),\alpha^{-1}R^\top( q_i - q_j) \rangle\right)\langle \alpha^{-2}R^\top \Sigma R (\alpha R^\top b_i),\alpha R^\top b_j\rangle \\
   &= \sum_{i,j}\exp\left(-\frac{1}{2}\langle  \Sigma^{-1}( q_i - q_j)),( q_i - q_j) \rangle\right)\langle  \Sigma b_i, b_j\rangle \\
   &= \lvert u\rvert_{V_\Sigma}^2
   \end{align*}
   which concludes the proof.
\end{proof}

\subsubsection{Adding the group of diffeomorphisms}
Following section \ref{Sec:diff_structure}, we can now define the half-Lie group of deformations as a semidirect product of the finite dimensional Lie group $\alpha\operatorname{-Isom}(\R^d)$ and of the group of $C^k_0$ diffeomorphisms.
\[
\alpha\operatorname{-Isom}(\R^d)\ltimes\operatorname{Diff}_{C_0^k}(\R^d)
\]
We recall that the group operations are defined by
\[
\left\{\begin{array}{l}
     (\alpha,R,T,\varphi)\cdot(\alpha',R',T',\varphi')= \left(\alpha\alpha',RR', \alpha R T'+T, \varphi_{(\alpha',R',T')}^{-1}\circ\varphi\circ\varphi_{(\alpha',R',T')}\circ\varphi'\right)\\
     (\alpha,R,T,\varphi)^{-1}=\left(\alpha^{-1},R^\top,-\alpha^{-1} R^\top T,\varphi_{(\alpha,R,T)}\circ\varphi^{-1}\circ\varphi_{(\alpha,R,T)}^{-1}\right) 
\end{array}
\right.
\]
and by proposition \ref{gga_struc_G} the family of half-Lie groups $\{\alpha\operatorname{-Isom}(\R^d)\ltimes\operatorname{Diff}_{C_0^k}(\R^d),\, k\geq1\}$ satisfies conditions (G.1-5).
\subsection{The coupled dynamic on the space of landmarks}
We can now define a new model using the group of deformations $\alpha\operatorname{-Isom}(\R^d)\ltimes\operatorname{Diff}_{C_0^k}(\R^d)$ in order to transport the shape and the anisotropy, avoiding the behaviour of Figure \ref{Ani:fig:toy_example_scaling_fixed}. Note that this model, and the idea of creating interactions between the shapes and the kernel that generates vector fields, is close to the general modular approach for diffeomorphic deformations introduced in \cite{gris2015sub}. In that work, the authors define a set of geometric descriptor to construct vector fields that transport the shape along with these geometric descriptors via its infinitesimal action. In this section, the global metric inducing the anisotropic Gaussian kernel can be interpreted as one of these geometric descriptors. In this part, for sake of readiness, we only deal with shapes encoded by landmarks.

\subsubsection{The shape space}
We consider a source shape $q_S\in \operatorname{Lmk}_n(\R^d)$ that we want to match to a given target $q_T \in \operatorname{Lmk}_n(\R^d)$. Moreover, we assume that we have a prior on the distribution of $q_S$ given by a symmetric definite positive matrix $\Sigma_S\in S_d^{++}$ representing the anisotropy of the shape. This matrix defines a metric, as stated previously, that we want to transport together with the shape.
Consequently, we define in the same fashion as section \ref{Sec:shape_space} the augmented shape space 
\[\tilde{\mathcal{Q}}= \alpha\operatorname{-Isom}(\R^d)\times S_d^{++}\times(\R^d)^n\] representing the orientation position, and a finer description of the shape with an anisotropic metric associated. The next step is to define the action of the semidirect product $\alpha\operatorname{-Isom}(\R^d)\ltimes\operatorname{Diff}_{C_0^k}(\R^d)$ on the manifold $\tilde{\mathcal{Q}}$. First, the group $\alpha\operatorname{-Isom}(\R^d)$ acts on the shape space $\tilde{\mathcal{Q}}$ by
\[
(\alpha,R,T)\cdot (\alpha',R',T',\Sigma,(q_i)_i) = \left(\alpha\alpha',RR', \alpha R T'+T, \alpha^2R\Sigma R^\top, (\alpha R q_i+T)_i\right)
\]
with infinitesimal action
\[
\xi_{\alpha,R,T,\Sigma,q_i}^{\alpha\operatorname{-Isom}}(s,A,\tau) = \big(s\alpha, AR, (s+A)T+\tau, 2s\Sigma+A\Sigma-\Sigma A, sq_i+Aq_i+sT+\tau\big).
\]
Moreover the group of diffeomorphisms simply acts on $\tilde{\mathcal{Q}}$ by transporting the points.
\[
\varphi\cdot (\alpha',R',T',\Sigma,(q_i)_i) =  (\alpha',R',T',\Sigma,(\varphi(q_i))_i)
\]
The combined action of the total group of deformations is thus given by
\begin{align*}
    \left((\alpha,R,T),\varphi \right)\cdot(\alpha',R',T',\Sigma,(q_i)_i) &\coloneqq (\alpha,R,T)\cdot \varphi\cdot (\alpha',R',T',\Sigma,(q_i)_i)\\
    &= \left(\alpha\alpha',RR', \alpha R T'+T, \alpha^2R\Sigma R^\top, (\alpha R \varphi(q_i)+T)_i\right)
\end{align*}
Differentiating this expression on the left at identity, we get the infinitesimal action
\begin{equation}
    \xi : \left\{\begin{array}{cll}
         \left(\alpha\text{-}\mathfrak{isom}(\R^d)\times C_0^k(\R^d,\R^d)\right)\times \tilde{\mathcal{Q}} &\longrightarrow & T\tilde{\mathcal{Q}}  \\
         (s,A,\tau,u), (\alpha,R,T,\Sigma,q) & \longmapsto & \big(s\alpha, AR, (s+A)T+\tau, 2s\Sigma+A\Sigma-\Sigma A,\\ & & sq+Aq+u(q)+sT+\tau\big)
    \end{array}
    \right.
\end{equation}

\subsubsection{The matching problem}
We suppose we are given a data attachment term $\mathcal{D}:\tilde{\mathcal{Q}}\to \R$. We recall that we are given a source shape $\bqS\in\operatorname{Lmk}_n(\R^d)$ together with a source metric $\Sigma_S\in S_d^{++}$. Note that, since the metric $\Sigma_S$ is transported and is used to compute the energy of the vector field transporting the shape $\vert u \vert^2$, the problem and the sub-Riemannian structure we are studying here is not right-invariant anymore. The induced structure on $\tilde{\mathcal{Q}}$ is given by the bundle $\overline{V} \hookrightarrow  T\tilde{\mathcal{Q}}$ such that for any $(\alpha,R,T,\Sigma,\bq)\in \tilde{\mathcal{Q}}$, we get 
\[
\overline{V}_{\alpha,R,T,\Sigma,\bq} = \alpha\text{-}\mathfrak{isom}(\R^d)\oplus V_\Sigma,
\]
where $V_\Sigma$ is the RKHS induced by the Gaussian kernel $k_D$, together with the bundle morphism $\xi:\overline{V}\to T\tilde{\mathcal{Q}}$ and the metric defined on $\overline{V}$ by 
\[
\langle (s,A,\tau,u),(s,A,\tau,u)\rangle_{\alpha,R,T,\Sigma,\bq} = \lvert s\rvert^2+\lvert A\rvert^2+\lvert\tau\rvert^2+\lvert u\rvert^2_{V_\Sigma}
\]
Moreover, even though this structure is not right-invariant, it depends here only on the metric $\Sigma$ and is still invariant by the action of diffeomorphism part.
We define, as in section \ref{Sec:var_pb}, the following variational problem
\begin{eqnarray}
    \label{Ani:var_pb_an_beforechgtvar}
    \inf_{(s_t,A_t,\tau_t,u_t) \in L^2([0,1],\overline{V})} J(s_t,A_t,\tau_t,u_t) &=& \frac{1}{2}\int_0^1\lvert s_t\rvert^2+\lvert A_t\rvert^2+\lvert\tau_t\rvert^2+\lvert u_t\rvert^2_{V_{\Sigma_t}} \, dt \\ & & \qquad \qquad \qquad+ \mathcal{D}(\alpha_1,R_1,T_1,\Sigma_1,\bq_1)\notag\\[0.5em]
     \text{ s.t }& &     \left\{
        \begin{array}{l}
            (\dot{\alpha}_t,\dot{R}_t,\dot{T}_t,\dot{\Sigma}_t,\dot{\bq}_t) = \xi_{\alpha_t,R_t,T_t,\Sigma_t,\bq_t}(s_t,A_t,\tau_t,u_t) \\
            (\alpha_0,R_0,T_0,\Sigma_0,\bq_0)=(1,Id,0,\Sigma_S,\bqS)      
        \end{array} 
        \right.\notag
\end{eqnarray}
To simplify this problem, we can actually consider the same change of variables as in \eqref{goat_change_variable},
\begin{equation*}
    \left\{\begin{array}{l}
         \tilde{\bq}_i = (\alpha,R,T)^{-1}\cdot \bq_i= \alpha^{-1}R^\top \bq_i - \alpha^{-1}R^\top T  \\
          \Tilde{\Sigma} = (\alpha,R,T)^{-1}\cdot \Sigma = \alpha^{-2}R^\top \Sigma R
    \end{array}
    \right.
\end{equation*}
which removes the action of the Lie group $\alpha\operatorname{-Isom}(\R^d)$ on the space of landmarks. The dynamic for these new variables is given by the infinitesimal action
\begin{equation}\label{Ani:tltxi}
    \tilde{\xi} : \left\{\begin{array}{cll}
         \left(\alpha\text{-}\mathfrak{isom}(\R^d)\times C_0^k(\R^d,\R^d)\right)\times \tilde{\mathcal{Q}} &\longrightarrow & T\tilde{\mathcal{Q}}  \\
         (s,A,\tau,u), (\alpha,R,T,\tilde{\Sigma},\tilde{q}) & \longmapsto & \big(s\alpha, AR, (s+A)T+\tau,0, \alpha^{-1}R^\top u(\alpha R\tilde{\bq}+ T)\big)
    \end{array}
    \right.
\end{equation}
In particular, we now see that the metric $\tilde{\Sigma}$ becomes a constant of the dynamic. In addition, by defining $\tilde{u} := \alpha^{-1}R^\top u(\alpha R\cdot+ T)$, proposition \ref{Ani:isom_RKHS} states that $\lvert\tilde{u}\rvert_{V_{\tilde{\Sigma}}}=\lvert u \rvert_{V_\Sigma}$  and we obtain the equivalent following matching problem (cf. proposition \ref{eq_var_pb}).
\begin{eqnarray}
    \label{Ani:var_pb_an_afterchgtvar}
    \inf_{(s_t,A_t,\tau_t,\tilde{u}_t) \in L^2([0,1],\alpha\text{-}\mathfrak{isom}(\R^d)\oplus V_{\Sigma_S})}& \tilde{J}(s_t,A_t,\tau_t,\tilde{u}_t) &= \frac{1}{2}\int_0^1\lvert s_t\rvert^2+\lvert A_t\rvert^2+\lvert\tau_t\rvert^2+\lvert \tilde{u}_t\rvert^2_{V_{\Sigma_S}} \, dt \\  & & +\mathcal{D}(\alpha_1,R_1,T_1,\Sigma_1,(\alpha_1,R_1,T_1)\cdot\tilde{\bq}_1)\notag
\end{eqnarray}
with dynamic given by
\begin{equation*}
    \left\{\begin{array}{l}
            (\dot{\alpha}_t,\dot{R}_t,\dot{T}_t) = (s_t\alpha_t,A_tR_t,(s_t+A_t)T_t+\tau_t) \\
            \tilde{\Sigma}_t=\Sigma_S\\
            \dot{\tilde{\bq}}_t=\tilde{u}_t(\tilde{\bq}_t) \\
            (\alpha_0,R_0,T_0,\Sigma_0,\tilde{\bq})=(1,I_d,0,\Sigma_S,\bqS)      
        \end{array} 
        \right.\notag
\end{equation*}
In the next proposition, we characterize the critical points of problem \eqref{Ani:var_pb_an_afterchgtvar}.
\begin{prop}[Hamiltonian characterization of \eqref{Ani:var_pb_an_afterchgtvar}]
The Hamiltonian $H : T^*\tilde{\mathcal{Q}}\to\R$ associated with the matching problem \eqref{Ani:var_pb_an_afterchgtvar} is given by :
    \begin{equation}
    \label{Ani:Hamil_Dtransport}
        H=\frac{1}{2}\left(\lvert p^\alpha \alpha+p^\tau T^\top\rvert^2+\lvert p^AR^\top+\operatorname{Skew}(p^\tau T^\top)\rvert^2+\lvert p^\tau \rvert^2+\left\lvert K_{\tilde{\Sigma}}\sum_i\delta_{\tilde{q}_i}^{ \tilde{p}_i}\right\rvert_{V_{\tilde{\Sigma}}}^2 \right)
    \end{equation}
where $\operatorname{Skew}(A)=\frac{1}{2}(A-A^\top)$ denotes the skew-symmetric of a matrix $A \in M_d(\R)$.
\end{prop}
\begin{proof}
We define the Hamiltonian corresponding to the problem \eqref{Ani:var_pb_an_afterchgtvar}
\begin{multline*}
H(\alpha,R,T,\tilde{\Sigma},\tilde{\bq},p^\alpha,p^A,p^{\tau},\tilde{p}^\Sigma,\tilde{\bp},s,A,\tau,\tilde{u})=(p^\alpha\mid s\alpha)+(p^A\mid AR)+(p^\tau\mid \tau + (A+s)T) + \\ \sum_i(\tilde{p}_i\mid \tilde{u}(\tilde{q}_i)) - 
        \frac{1}{2} \left(\lvert s\rvert^2+\lvert A\rvert^2 + \lvert\tau \rvert^2 + \lvert \tilde{u}\rvert_{\tilde{\Sigma}}^2\right)  
\end{multline*}
The condition on the controls
\begin{equation*}
    (\partial_sH,\partial_A H, \partial_{\tau} H, \partial_{\tilde{u}} H) = 0
\end{equation*}
therefore gives us
\[
\left\{\begin{array}{l}
     s=p^\alpha \alpha + T^{\top} p^\tau  \\
     A=p^AR^\top + \operatorname{Skew}(p^\tau T^\top) \\
     \tau=p^\tau \\
     \tilde{u} = K_{\tilde{\Sigma}} \sum_i\delta_{\tilde{q}_i}^{\tilde{p}_i}
\end{array}\right.
\] and the result follows.
\end{proof}

\paragraph{Numerical result} We experiment our new model on the same source and target of Figure \ref{Ani:fig:source_target}, using the same initial metric $\Sigma_S=\begin{bmatrix}
1 & 0  \\
0        & 0.1
\end{bmatrix}$, meaning we favor deformations along the $x$-axis. However, note that in this new model, any rotation $R$ of the source will also rotate the axis of the anisotropy which will induce a new metric $R\Sigma R^\top$. Consequently, the algorithm cannot rotate the shape to align it with the favored axis before stretching, as illustrated previously in Figure \ref{Ani:fig:toy_example_scaling_fixed}, since the axis will also rotate. Results of this new model are reported in Figure \ref{Ani:fig:toy_example_scaling_non_fixed}, where we observe that no rotations are applied and the stretching is done along the $y$-axis even though the energy cost is higher.
\begin{figure}[ht]
\centering
\includegraphics{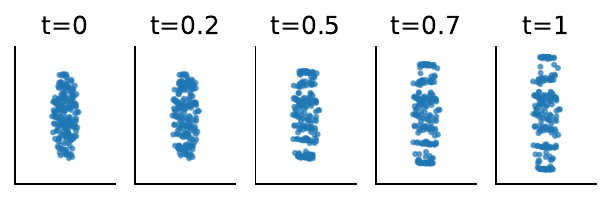}
\caption{\emph{Performing jointly rigid and non rigid registration minimizing problem \eqref{Ani:var_pb_an_afterchgtvar}, with an anisotropic kernel.} This experiment illustrates a matching task between the source (blue) and the target (red). The matching is performed through deformations generated by rotations and diffeomorphic deformation induced by an anisotropic Gaussian kernel with scale $\Sigma_S=\begin{bmatrix}
1 & 0  \\
0        & 0.1
\end{bmatrix}$.}
\label{Ani:fig:toy_example_scaling_non_fixed}
\end{figure}

\subsubsection{Invariance property of the non rigid part of the Hamiltonian}
Following section \ref{reduction}, and by proposition \ref{Ani:isom_RKHS} we see that the term in the Hamiltonian \eqref{Ani:Hamil_Dtransport} corresponding to the non-rigid part 
\[
H(\tilde{\Sigma},\tilde{\bq},\tilde{p}^\Sigma,\tilde{\bp}) = \frac{1}{2}\sum_{i,j} \left\langle \tilde{p}_i, k_{\tilde{\Sigma}}(\tilde{q}_i,\tilde{q}_j)\tilde{p}_j \right\rangle.
\]
is invariant by the action of the Lie group $\alpha\operatorname{-Isom}(\R^d)$.
This kernel is normalized in the sense that it allows to define Hamiltonians that are invariant under the action of rigid motions $\alpha\operatorname{-Isom}(\R^d)$:
\begin{prop}[$\alpha\operatorname{-Isom}(\R^d)$-invariance of the Hamiltonian]
    The group $\alpha\operatorname{-Isom}(\R^d)$  defines a canonical Hamiltonian action (cf. \ref{App:Ham_action}) on the symplectic manifold $T^*(S_d^{++}\times \operatorname{Lmk}_n)$, with momentum map 
    \begin{equation}
        \mu : \app{T^*(S_d^{++}\times \operatorname{Lmk}_n)}{\alpha\text{-}\mathfrak{isom}(\R^d)}{(p^\Sigma,p)}{\left(\sum_i\langle p_i,q_i\rangle + p^\Sigma \Sigma, \sum_i\operatorname{Skew}(p_iq_i^\top) + \operatorname{Skew}(p^\Sigma \Sigma), \sum_ip_i\right)}
    \end{equation}
    Moreover, the Hamiltonian H defined previously is invariant by the action of $\alpha\operatorname{-Isom}(\R^d)$ and the momentum map is a constant of the motion.
\end{prop}
\begin{proof}
    The action of $\alpha\operatorname{-Isom}(\R^d)$ on $T^*(S_d^{++}\times \operatorname{Lmk}_n))$ is given by 
    \[(\alpha,R,T)\cdot(\Sigma,q,p^\Sigma,p) = \left(\alpha^2 R\Sigma R^\top,\alpha Rq+ T,\alpha^{-2}Rp^\Sigma R^\top,\alpha^{-1}Rp\right)\] 
    which is the cotangent lift of the action of $\alpha\operatorname{-Isom}(\R^d)$ on $\tilde{\mathcal{Q}}$, and therefore is a Hamiltonian action.
The rest of the proof follows from proposition \ref{Ani:isom_RKHS} and from Noether theorem.
\end{proof}

\subsubsection{Transport of the metric by diffeomorphisms}
In the previous variational problem \eqref{Ani:var_pb_an_beforechgtvar} (or equivalently \eqref{Ani:var_pb_an_afterchgtvar}),  the metric $\Sigma\in S_d^{++}$ representing the anisotropy of the shape was transported by the action of the group $\alpha\operatorname{-Isom}(\R^d)$. This section presents a method to transport this metric by the group of diffeomorphisms too.

The tricky part is that there is no natural action of diffeomorphisms on the space $S_d^{++}$ of metrics in $\R^d$. However, even though there is no such group action on this manifold, it is still possible to define an infinitesimal action of the space of $C_0^k$-vector fields on $S_d^{++}$. We assume that there exists a matrix 
\[
M : \app{ (\R^d)^n \times C_0^k(\R^d,\R^d)}{\R^{d\times d}}{(q,u)}{M_q(u)}
\]
such that the infinitesimal transport of a shape ($\bq,\Sigma)$ by a vector field $u\in C_0^k(\R^d,\R^d)$ is defined by
\[
\xi_{\Sigma,\bq}^{\Diff}= \left(M_{\bq}(u)\Sigma+\Sigma M_{\bq}(u)^\top,u(q_i)_i\right)
\]
\begin{rem}
    Here, the connection with the modular framework \cite{gris2015sub} becomes clearer. Indeed, we use the shape $q$ and the control $u$ to generate a matrix that transports the metric $D$ representing anisotropy in the shape.
\end{rem}
\begin{exam}
 $M : (q,u) \mapsto M_q(u) := \frac{1}{N}\sum_{i\leq N}du(q_i)$
represents the mean of the jacobian on the shape $(q_i)_{i\leq N}$ and thus summarizes how the vector field $u$ acts on the anisotropy.
\end{exam}
Now taking the sum of the two infinitesimal actions $\xi^{\alpha\operatorname{-Isom}}$ and $\xi^{\Diff}$, we can thus modify the previous infinitesimal action and define
\begin{equation*}
    \xi : \left\{\begin{array}{cll}
         \left(\alpha\text{-}\mathfrak{isom}(\R^d)\times C_0^k(\R^d,\R^d)\right)\times \tilde{\mathcal{Q}} &\longrightarrow & T\tilde{\mathcal{Q}}  \\
         (s,A,\tau,u), (\alpha,R,T,\Sigma,q) & \longmapsto & \big(s\alpha, AR, (s+A)T+\tau, 2s\Sigma+(A+M_{q}(u))\Sigma- \\ & &\Sigma(A-M_{q}(u)^\top), sq+Aq+u(q)+sT+\tau\big)
    \end{array}
    \right.
\end{equation*}
The same change of variables
\begin{equation*}
    \left\{\begin{array}{l}
         \tilde{q}_i = (\alpha,R,T)^{-1}\cdot q_i= \alpha^{-1}R^\top q_i - \alpha^{-1}R^\top T  \\
          \Tilde{\Sigma} = (\alpha,R,T)^{-1}\cdot \Sigma = \alpha^{-2}R^\top \Sigma R\\
          \tilde{u} = \alpha^{-1}R^\top u(\alpha R\cdot +T)
    \end{array}
    \right.
\end{equation*}
leads to a new dynamic associated with the infinitesimal action
\begin{equation*}
    \tilde{\xi} : \left\{\begin{array}{cll}
         \left(\alpha\text{-}\mathfrak{isom}(\R^d)\times C_0^k(\R^d,\R^d)\right)\times \tilde{\mathcal{Q}} &\longrightarrow & T\tilde{\mathcal{Q}}  \\
         (s,A,\tau,\tilde{u}), (\alpha,R,T,\tilde{\Sigma},\tilde{q}) & \longmapsto & \big(s\alpha, AR, (s+A)T+\tau, \tilde{M}_{\alpha,R,T,\tilde{q}}(\tilde{u})\tilde{\Sigma}+ \\ & & \tilde{\Sigma}\tilde{M}_{\alpha,R,T,\tilde{q}}(\tilde{u})^\top ,\, \tilde{u}(\tilde{q})\big)
    \end{array}
    \right.
\end{equation*}
where $\tilde{M}:\tilde{\mathcal{Q}}\times C_0^k(\R^d,\R^d)\to \R^{d\times d}$ is given by $\tilde{M}_{\alpha,R,T,\tilde{q}}(\tilde{u}) = R^\top M_{\alpha R\tilde{q}+T}\left(d\varphi_{(\alpha,R,T)}\tilde{u}\circ\varphi_{(\alpha,R,T)^{-1}} \right)R$ 
In such case, the new variable $\tilde{\Sigma}$ is not a constant of the dynamic anymore (as in \eqref{Ani:tltxi}) since it is transported by the vector field $u$.
The matching problem associated to this framework leads to the minimization of the following new functional
\begin{eqnarray}
    \label{Ani:var_pb_an_afterchgtvar_fin}
    \inf_{(s,A,\tau,u) \in L^2([0,1],\overline{V})}& \tilde{J}(s,A,\tau,\tilde{u}) &= \frac{1}{2}\int_0^1\lvert s\rvert^2+\lvert A\rvert^2+\lvert\tau\rvert^2+\lvert \tilde{u}\rvert^2_{V_{\tilde{\Sigma}}} \, dt  \\ & & +\mathcal{D}(\alpha_1,R_1,T_1,\Sigma_1,(\alpha_1,R_1,T_1)\cdot\tilde{\bq}_1) \notag
\end{eqnarray}
where $(\Tilde{\Sigma}_t,R_t,T_t,\tilde{q}_t)$ satisfies the dynamic
\[
\left\{\begin{array}{l}     
      (\dot{\alpha}_t,\dot{R}_t,\dot{T}_t,\dot{\Tilde{\Sigma}}_t,\dot{\tilde{\bq}}_t) =  \tilde{\xi}_{(\tilde{\Sigma}_t,R_t,T_t,\tilde{\bq}_t)}(s_t,A_t,\tau_t,u_t),\\
      (\alpha_0,R_0,T_0,\Tilde{\Sigma}_0,\tilde{q}_0) = (1,I_d,0,\Sigma_S,\bqS)
\end{array}\right.
\]
In the next proposition, we compute the Hamiltonian associated with the variational problem \eqref{Ani:var_pb_an_afterchgtvar_fin}:

\begin{prop}
    The Hamiltonian $H : T^*\tilde{\mathcal{Q}}\to\R$ associated with the matching problem \eqref{Ani:var_pb_an_afterchgtvar_fin} is given by :
    \begin{multline}
        H=\frac{1}{2}(\lvert p^\alpha \alpha+p^\tau  T^\top\rvert^2+\lvert p^AR^\top+\operatorname{Skew}( T^\top p^{\tau})\rvert^2+\lvert p^\tau\rvert^2)+ \\
        \frac{1}{2} \left\lvert K_{\Tilde{\Sigma}} \left(\sum_i\delta_{\tilde{q}_i}^{\tilde{p}_i} + 2M^*_{\alpha,R,T,\tilde{\bq}}(\tilde{p}^\Sigma\tilde{\Sigma})\right)\right\rvert_{V_{\tilde{D}}}^2  
    \end{multline}
\end{prop}
\begin{proof}
We define the Hamiltonian corresponding to the problem \eqref{Ani:var_pb_an_afterchgtvar_fin}
\begin{multline*}
H(\alpha,R,T,\tilde{\Sigma},\tilde{\bq},p^\alpha,p^A,p^{\tau},\tilde{p}^\Sigma,\tilde{\bp},s,A,\tau,\tilde{u})=(p^\alpha\mid s\alpha)+(p^A\mid AR)+(p^\tau\mid \tau + (A+s)T) + \\ (\tilde{p}^\Sigma\mid  M_{\alpha,R,T,\tilde{q}}(\tilde{u})\Tilde{\Sigma}+ \Tilde{\Sigma}M_{\alpha,R,T,\tilde{q}}(\tilde{u})^\top)+ \sum_i(\tilde{p}_i\mid \tilde{u}(\tilde{q}_i)) - 
        \frac{1}{2} \left(\lvert s\rvert^2+\lvert A\rvert^2 + \lvert\tau \rvert^2 + \lvert \tilde{u}\rvert_{\Tilde{\Sigma}}^2\right)  
\end{multline*}
The condition on the controls
\begin{equation*}
    (\partial_sH,\partial_A H, \partial_{\tau} H, \partial_{\tilde{u}} H) = 0
\end{equation*}
therefore gives us
\[
\left\{\begin{array}{l}
     s=p^\alpha \alpha + T^{\top} p^\tau  \\
     A=p^AR^\top + \operatorname{Skew}(p^\tau T^\top) \\
     \tau=p^\tau \\
     \tilde{u} = K_{\Tilde{\Sigma}} \left(\sum_i\delta_{\tilde{q}_i}^{\tilde{p}_i} + 2M^*_{\alpha,R,T,\tilde{\bq}}(p^\Sigma\tilde{\Sigma})\right)
\end{array}\right.
\] and the result follows.
\end{proof}

\paragraph{Acknowledgments.} The authors would also like to express their gratitude to Sylvain Arguillère (University of Lille), Anton François (ENS Paris-Saclay), Barbara Gris (Sorbonne University), Irène Kaltenmark (Université Paris Cité) and Alain Trouvé (ENS Paris-Saclay) for their insightful comments and valuable discussions that contributed to improve the quality of this work.

\addcontentsline{toc}{section}{References}
\printbibliography

\appendix
\section{Proof of proposition \ref{relation_moment}} \label{proof_moments}
This section is devoted to the proof of proposition \ref{relation_moment}. Let $q_S\in\mathcal{Q}$ be a template shape, and we consider $(e_G,q_S)\in \tilde{\mathcal{Q}}=G\times \mathcal{Q}$ the corresponding shape. We recall the two Hamiltonians on $T^*\tilde{\mathcal{Q}}= T^*G\oplus T^*\mathcal{Q}$.
\[
    H(g,q,p^g,p) =\frac{1}{2} \lvert K_{\mathfrak{g}} ((d_{e_G}R_g)^* p^{g} + \xi_q^{\mathfrak{g}*}p) \rvert^2_{\mathfrak{g}} + \frac{1}{2} \lvert K_V \xi_{q}^{*}p \rvert^2_V
\]
and after the change of variable $\tilde{q}=g^{-1}\cdot q$
\[
\tilde{H}(g,\tilde{q},\tilde{p}^g,\tilde{p}) = \frac{1}{2}\vert K_{\mathfrak{g}} (d_{e_G}R_g)^*\tilde{p}^g\vert_\mathfrak{g}^2 + \frac{1}{2} \vert K_V(d_{\id}\rho_g)^*\xi_{\tilde{q}}^* \tilde{p}\vert_V^2
\]
We will prove that both these Hamiltonians lead to equivalent dynamic. Let $(p^g_0,p_0),(\tilde{p}^g_0,\tilde{p}_0)\in T_{(e_G,q_S)}^*\tilde{\mathcal{Q}}$ initial covectors such that $\tilde{p}^{g}_0 = p^{g}_0 + \xi_{q_S}^{\mathfrak{g}*}p_0$ and $\tilde{p}_0 = \big(\partial_q ( e_G \cdot q )\vert_{q=q_S}\big)^*p_0$ . Let $(g_t,q_t)$ (resp. $(\tilde{g}_t,\tilde{q}_t)$) be the Hamiltonian flow of $H$ (resp. $\tilde{H}$). By definition, the Hamiltonian $H$ is induced by the right invariant metric of $\mathcal{G}^k$ and the action of $\mathcal{G}^k$ on the space $\mathcal{Q}$. In particular \cite[§.5]{gga}, it leads to a momentum map $m:T^*\mathcal{\tilde{Q}}\to (T_e\mathcal{G}^k)^*$ defined by
\[
(m(g,q,p^\mathfrak{g},p)\mid X,v) = (p^\mathfrak{g}\mid d_{e_G}R_g(X) ) + (p\mid \xi^\mathfrak{g}_q(X) + \xi_q(v))
\]
where $(g,q,p^\mathfrak{g},p)\in T^*\mathcal{\tilde{Q}}$, and $(X,v)\in \mathfrak{g}\times C_0^k(\R^d,\R^d)$, that defines the momentum trajectory $t\mapsto m_t =m(g_t,q_t,p^g_t,p_t)$ associated to the Hamiltonian flow. The definition of the momentum map here is particularly interesting since it determines the Hamiltonian flow of $H$, and it lives directly in the dual of tangent space at identity of $\mathcal{G}^k$. Its dynamic follows the sub-Riemannian Euler-Poincaré equation \cite[Theorem 5.2]{gga}, that we recall in its integrated form:
\begin{equation}
\label{app:int_epdiff}
    m_t = \operatorname{Ad}_{(g_t,\varphi_t)^{-1}}^*(m_0),
\end{equation}
where $m_0=(p^\mathfrak{g}_0,\xi_{q_S}^{*}p_0)\in \mathfrak{g}^*\oplus C_0^k(\R^d,\R^d)^*$. In the next step, we prove that the Hamiltonian $\tilde{H}$, after the change of variable, can also be written as the Hamiltonian associated to sub-Riemannian metric on $\tilde{\mathcal{Q}}$ induced by an action of the group $\mathcal{G}^k$. Indeed, we define, for $(g,\varphi)\in \mathcal{G}^k$ and $(g',\tilde{q})\in \tilde{\mathcal{Q}}$, the action 
\[
\tilde{A}\left((g,\varphi), (g',\tilde{q})\right) = \left(gg',\rho_{g'}(\varphi)\cdot \tilde{q}\right).
\]
We prove next that this action satisfies conditions of proposition \ref{tilde_q_shape_space}
\begin{lemma}[Second action on $\tilde{\mathcal{Q}}$] \label{tilde_q_shape_space2}
    The action $\tilde{A}$ of the half-Lie group $\mathcal{G}^k$ on $\tilde{\mathcal{Q}}$ satisfies the following conditions 
    \begin{enumerate}
        \item The action $\tilde{A} : \mathcal{G}^k \times \tilde{\mathcal{Q}}\to\tilde{\mathcal{Q}}$ is continuous.
        \item For $(h,\tilde{q}) \in \tilde{\mathcal{Q}}$, the mapping $\tilde{A}_{h,\tilde{q}}:(g,\varphi)\mapsto (gh,g\cdot(\varphi\cdot \tilde{q}))$ is smooth, and we denote $\tilde{\xi}_{h,\tilde{q}}=d_{(e_G,\id)}\tilde{A}_{h,\tilde{q}}$ the induced infinitesimal action.
        \item For $l>0$, the mappings \[\tilde{A} : \app{\mathcal{G}^{k+l} \times \tilde{\mathcal{Q}}}{\tilde{\mathcal{Q}}}{(g,\varphi),(h,\tilde{q})}{(gh,g\cdot (\varphi \cdot \tilde{q}))} \text{ and } \tilde{\xi} : \app{T_{(e_G,\id)}\mathcal{G}^{k+l} \times \tilde{\mathcal{Q}}}{T\tilde{\mathcal{Q}}}{(X,u),(h,\tilde{q})}{(Xh, X \cdot \tilde{q} + u\cdot \tilde{q})}\] are $C^l$. 
    \end{enumerate}
    Moreover any curve $(g_t,\tilde{q}_t,p^g_t,p_t)$ that is the Hamiltonian flow of $\tilde{H}$ in $T^*\tilde{\mathcal{Q}}$, with initial condition $(g_0,\tilde{q}_S)=(e_G,q_S)$, can be lifted to a curve $(g'_t,\varphi_t,p'^g_t,p^\varphi_t)$ in the space $T^*\tilde{\mathcal{G}}^k$ such that 
    \begin{enumerate}
        \item $(p'^g_0,p^\varphi_0)=(p^g_0,\xi^{*}_{\tilde{q}}p^{\varphi}_0)$
        \item $(g_t,\tilde{q}_t)=(g'_t,\varphi_t)\cdot(e_G,q_S)$, and in particular $g_t=g'_t$.
    \end{enumerate}    
\end{lemma}
\begin{proof}
The first part of the proof is similar to the proof of proposition \ref{tilde_q_shape_space}, since the regularity conditions of the action $\tilde{A}$ follows from the structure of the semi-direct product $\mathcal{G}^k$ and on the regularity asumptions on all the different actions. Moreover, as we saw, the Hamiltonian $\tilde{H}$ comes from the pre-Hamiltonian
\[
\tilde{H}(g',\tilde{q},p^g,p,X,v) = \left(p^g,p\mid \tilde{\xi}_{g',\tilde{q}}(X,v)\right) - \frac{1}{2}(\lvert X\rvert_\mathfrak{g}^2 + \lvert v\rvert_V^2)
\] so that the Hamiltonian $\tilde{H}$ is induced by the right invariant metric on $\mathcal{G}^k$ and the action of the group $\mathcal{G}^k$ on $\tilde{Q}$. The last part of the lemma follows then from \cite[Theorem 5.2]{gga}.
\end{proof}
    In particular, this action also induces a momentum map $\tilde{m}:T^*\tilde{\mathcal{Q}}\to T_e^*\mathcal{G}^k$ defined by
    \[
    \tilde{m}(g,\tilde{q},p^g,\tilde{p})(X,v) = (p^g \vert d_eR_gX) + (\tilde{p} \vert \xi_{\tilde{q}}d_{\id} \rho_g v),
    \]
    and if $(g_t,\tilde{q}_t,p^g_t,\tilde{p}_t)$ is the Hamiltonian flow of $\tilde{H}$, then the momentum $m_t = \tilde{m}(g_t,\tilde{q}_t,p^g_t,\tilde{p}_t)$ also satisfies the sub-Riemannian Euler-Poincaré-Arnold equation \eqref{app:int_epdiff}, and this totally determines the flow $(g_t,\tilde{q}_t,p^g_t,\tilde{p}_t)$. In particular if the initial momenta $m_0$ and $\tilde{m}_0$ associated to a flow of $H$ and $\tilde{H}$ are equal, then the momenta $m_t$ and $\tilde{m}_t$ are equal for all time and they lead to the same curve in $T^*\mathcal{G}^k$. We get therefore the following result.

\begin{lemma} \label{app:relation_momentum}
Let $(p^g_0,p_0),(\tilde{p}^g_0,\tilde{p}_0)\in T_{(e_G,q_S)}^*\tilde{Q}$ initial covectors such that $\tilde{p}^{g}_0 = p^{g}_0 + \xi_{q_S}^{\mathfrak{g}*}p_0$ and $\tilde{p}_0 = \partial_{2}A(g,q_S) ^*p_0$. Let $(g_t,q_t)$ (resp. $(\tilde{g}_t,\tilde{q}_t)$) be the Hamiltonian flow of $H$ (resp. $\tilde{H}$). Then for all time $t$, 
    \begin{equation} 
     \left\{
        \begin{array}{l}
            g_t = \tilde{g}_t \\
            q_t = g_t\cdot \tilde{q}_t \\
            (d_{e_G}R_{g_t})^*p_t^{g} + \xi_{q_t}^{\mathfrak{g}*}p_t = (d_{e_G}R_{g_t})^*\tilde{p}_t^g  \\
            \xi_{q_t}^{*}p_t = (d_{\id} \rho_{g_t})^*\xi_{\tilde{q}_t}^{*}\tilde{p}_t 
        \end{array} 
        \right.        
    \end{equation}
\end{lemma}
\begin{proof}
    We use the momenta trajectories to prove this result. We define the momentum $m_t = m(g_t,q_t,p^g_t,p_t)$ associated to the flow of $H$, and $\tilde{m}_t = \tilde{m}(\tilde{g}_t,\tilde{q}_t,\tilde{p}^g_t,\tilde{p}_t)$ associated to $\tilde{H}$. In particular, we have $m_0 = (p^g_0 +\xi_{q_S}^\mathfrak{g*}p_0,\xi_{q_S}^{*}p_0)$, and $\tilde{m}_0=(p^g_0,\xi_{q_S}^{*}\tilde{p_0})$, so that, by asumptions
    \[
    m_0 = \tilde{m}_0
    \]
    But then, by \cite[5.4]{gga}, we have for all $t$,    $m_t = \tilde{m}_t$ which proves the lemma.
\end{proof}
Note that this lemma is not sufficient to prove proposition \ref{relation_moment}. Indeed, let $A(g,q)$ denotes the action of $g \in G$ on $q \in \mathcal{Q}$, we can notice that $d_{\id}\rho_{g_t}\xi_{\tilde{q}_t}^{*}=\xi_{q_t}^{*}\partial_2 A(g^{-1}_t,q_t)^*$ thanks to the compatibility condition. Consequently the second equation can be expressed by :
    \begin{equation*}
        \xi_{q_t}^{*}p_t =\xi_{q_t}^{*}\partial_2 A(g_t^{-1},q_t)^*\tilde{p}_t
    \end{equation*}
However, we cannot deduce directly an equality between the moments since the operator $\xi_q$ is not necessarily invertible. Thus, we will prove this relation by direct calculus. 
    
\begin{prop}[Influence of the change of variable on covectors] \label{app:relation_moment}
    Let $A : (g,q) \rightarrow A(g,q)$ be the action of $G$ on $\mathcal{Q}$, $\xi^\mathfrak{g}$ its associated infinitesimal action. Let $(p^g_0,p_0),(\tilde{p}^g_0,\tilde{p}_0)\in T_{(e_G,q_S)}^*\tilde{\mathcal{Q}}$ initial covectors such that $\tilde{p}^{g}_0 = p^{g}_0 + \xi_{q_S}^{\mathfrak{g}*}p_0$ and $\tilde{p}_0 = \partial_{2}A(e_G,q_S) ^*p_0$. Let $(g_t,q_t)$ (resp. $(\tilde{g}_t,\tilde{q}_t)$) be the Hamiltonian flow of $H$ (resp. $\tilde{H}$). Then for all time $t$,
    \begin{equation}
     \left\{
        \begin{array}{l}
            \tilde{g}_t=g_t\\
            \tilde{q}_t= g_t\cdot q_t\\
            \tilde{p}^{g}_t = p^{g}_t + (d_{e_G}R_{g_t^{-1}})^* \xi_{q_t}^{\mathfrak{g}*}p_t \\
            \tilde{p}_t = \partial_2A(g_t,\tilde{q}_t)^*p_t 
        \end{array} 
        \right.        
    \end{equation}
\end{prop}

\begin{proof}
By lemma \ref{app:relation_momentum}, we directly get 
\begin{equation*}
     \left\{
        \begin{array}{l}
            \tilde{g}_t=g_t\\
            \tilde{q}_t= g_t\cdot q_t\\
            \tilde{p}^{g}_t = p^{g}_t + (d_{e_G}R_{g_t^{-1}})^* \xi_{q_t}^{\mathfrak{g}*}p_t\\
            \xi_{q_t}^{*}p_t =\xi_{q_t}^{*}\partial_2 A(g_t^{-1},q_t)^*\tilde{p}_t
        \end{array} 
        \right.        
\end{equation*}
Since $d_eR_g$ is invertible, we get the equality 
\[\tilde{p}^{g}_t = p^{g}_t + (d_{e_G}R_{g_t^{-1}})^* \xi_{q_t}^{\mathfrak{g}*}p_t.\]
It remains to prove the equality on the covectors $p_t$ and $\tilde{p}_t$. Let $A_t =\partial_2 A(g^{-1}_t,q_t)$, be an invertible operator with inverse $A_t^{-1} = \partial_2 A(g_t,\tilde{q}_t)$. In particular, we can relate this operator to the infinitesimal action.
\begin{equation} \label{relation_inf_action}
\xi_{\tilde{q}_t}(\tilde{v}_t) = A_t \xi_{q_t}(v_t) \quad \mbox{ or equivalently } \quad \xi_{q_t}(v_t) = A_t^{-1}  \xi_{\tilde{q}_t}(\tilde{v}_t)    
\end{equation}
We will prove the first equality of \ref{app:relation_moment} by showing that $\tilde{p}_t$ and $A_t^{-1*}p_t$ follow the same dynamic. For $\delta q \in T_{\tilde{q}}Q$
\begin{equation}\label{proof_eq_1}
    \partial_t (A_t^{-1*}p_t \vert \delta q)    = -(p_t \vert (\partial_q \xi_{q_t}(v_t) +  \partial_q \xi_{q_t}^{\mathfrak{g}}(X_t))  A_t^{-1}\delta q) +  (p_t \vert \partial_t A_t^{-1}\delta q)
\end{equation}
where the time derivative of $ A_t^{-1}$ is 
\begin{equation*}
   \partial_t A_t^{-1}= \partial^2_{1,2} A(g_t,\tilde{q}_t)d_{e_G}R_{g_t}(X_t) + \partial^2_{2,2} A(g_t,\tilde{q}_t) \dot{\tilde{q}}_t
\end{equation*}
By noticing that  $\partial_1 A(g_t,\tilde{q}_t) d_{e_G} R_{g_t}(X_t)= \xi_{g_t q_t}^{\mathfrak{g}}(X_t)$, it follows that
\begin{equation*}
    \partial_{1,2}^2 A(g_t,\tilde{q}_t) T_e R_{g_t}(X_t)= \partial_{q} \xi_{\tilde{q}_t}^{\mathfrak{g}}(X_t) A_t^{-1}
\end{equation*}
such that we can simplify the previous expression :
\begin{equation*}
   \partial_t A_t^{-1} = \partial_{q} \xi_{\tilde{q}_t}^{\mathfrak{g}}(X_t) A_t^{-1} + \partial^2_{2,2} A(g_t,\tilde{q}_t)\dot{\tilde{q}}_t
\end{equation*}
Then, by substition in Eq.\ref{proof_eq_1}, 
\begin{equation*}
    \partial_t (p_t \vert A_t^{-1}\delta q)   = (p_t \vert \partial^2_{2,2} A(g_t,\tilde{q}_t) (\dot{\tilde{q}}_t,\delta q)- \partial_q \xi_{q_t}(v_t)   A_t^{-1}\delta q  )
\end{equation*}
and using the relations from Eq. \ref{relation_inf_action}, we can deduce that
\begin{equation*}
    \partial_q \xi_{q_t} (v_t) A_t^{-1}=\partial_{2,2}^2 A(g_t,\tilde{q}_t)\dot{\tilde{q}}_t+ A_t^{-1}  \partial_q (\xi_{\tilde{q}_t}(\tilde{v}_t ))
\end{equation*}
We can conclude that 
\begin{equation*}
 \partial_t (A_t^{-1*}p_t \vert \delta q)   = -(  \partial_q (\xi_{\tilde{q}_t}(\tilde{v}_t ))^*A_t^{-1*}p_t \vert \delta q ) 
\end{equation*}
which leads to 
\begin{equation*}
    \frac{d}{dt}(A_t^{-1*} p_t) = - \partial_q (\xi_{\tilde{q}_t}(\tilde{v}_t ))^*(A_t^{-1*}p_t)
\end{equation*}
Finally, the result follows by the uniquess of the solution under equality of initial conditions thanks to Picard-Lindelöf theorem.

\end{proof}

\begin{rem}

Even though, the relation between $p$ and $\tilde{p}$ corresponds to the lift of the action $A_{g^{-1}}(q):=A(g^{-1},q)$ on the cotangent bundle $T^*Q$, the mapping 
\begin{equation*}
    \Phi : \app{T^*(G \times Q)}{T^*(G \times Q)}{(g,q,p^{\mathfrak{g}},p)}{(g,g^{-1}\cdot q, p^{\mathfrak{g}} + (d_{e_G}R_{g^{-1}})^* \xi_q^{\mathfrak{g}*}p,\partial_2A(g,g^{-1}\cdot q)^*p)}
\end{equation*} is not a symplectomorphism.
\end{rem}

\section{Symplectic geometry}
\label{App:symplecti_geom}
We remind some basic definitions and properties of symplectic geometry to have a better understanding of the notions involved in section \ref{reduction}. Further definitions and properties can be found in \cite{sympl_geom_dasilva}.

\begin{defi}[Symplectic manifold]
A symplectic manifold $(M,\omega)$ is a smooth manifold $M$ associated with a closed, non-degenerate 2-form $\omega$.
    
\end{defi}

A classic example of symplectic manifold is the cotangent bundle of a smooth manifold $T^*\mathcal{Q}$, endowed with the canonical symplectical form. This example is the one developed in section \ref{reduction} where the symplectic reduction is performed on the shape space $\mathcal{Q}$. In applications, the shape space can be finite-dimensional (landmarks, discrete curve etc.) or infinite-dimensional (images, smooth curves etc.) which implies different requirement on the symplectic form. Indeed, for $(q,p) \in T^*\mathcal{Q}$ the musical map $ \omega((q,p),\cdot) : (q',p') \in T^*\mathcal{Q} \mapsto \omega((q,p),(q',p')) \in \R$ needs to be an isomorphism in the finite-dimensional setting, which is often a too strong requirement for infinite-dimensional manifolds. In this case, we will rather define a notion of weakly symplectic form which only requires the injectivity of the musical map.  For the sake of simplicity, we only present the finite-dimensional setting although the stated results can be adapted to infinite-dimensional framework (see \cite{diez2024symplectic}). Note that, if $M$ is finite-dimensional, then $\operatorname{dim}(M)$ is even.

\begin{defi} \label{def:symplecto}
    Let $(M_1,\omega_1)$ and $(M_2,\omega_2)$ be two symplectic manifolds. A symplectomorphism $\varphi : M_1 \to M_2$ is a diffeomorphism that preserves the symplectic form $\varphi^* \omega_2 = \omega_1$ where $\varphi^*$ is the pullback of $\varphi$.
\end{defi}

In the context of variational problem, the equations of motion can be described through a Hamiltonian function $H : M \rightarrow \R$. 

 \begin{defi}[Hamiltonian vector field]
    Let $H : M \rightarrow \R$ be a Hamiltonian function. The unique vector field $X_H$ on $M$ such that $\iota_{X_H}\omega = dH$ is called the Hamiltonian vector field associated to $H$.
 \end{defi}
The Hamiltonian vector field associated to a Hamiltonian $H:M\to\R$ is also called simply the symplectic gradient of $H$, and is denoted by $\nabla^\omega H$. 
A consequence is that the Hamiltonian is constant on integral curve of $X_H$. \\
\begin{exam}
    Let $M=\R^{2n}$ with coordinates $(q_1,...,q_n,p_1,...,p_n)$ and canonical symplectic form $\omega = \sum dq_i \wedge dp_j$. Then, the differential of the Hamiltonian $H$ is given by 
\begin{equation*}
    dH = \sum (\frac{\partial H}{\partial q_i} dq_i - \frac{\partial H}{\partial p_i} dp_i)
\end{equation*}

which corresponds to the Hamiltonian vector field 

\begin{equation*}
    \nabla^\omega H= \sum (\frac{\partial H}{\partial p_i} \frac{\partial}{\partial q_i} - \frac{\partial H}{\partial q_i} \frac{\partial}{\partial p_i})
\end{equation*}

Therefore, if we consider $(q(t),p(t))$  integral curves of $\nabla^{\omega}H$, they satisfy the Hamiltonian equations :

\begin{equation*}
    \dot{q}_i(t) = \frac{\partial H}{\partial p_i}, \quad \dot{p}_i(t) = -\frac{\partial H}{\partial q_i} 
\end{equation*}
\end{exam}
Reduction theorem and Noether's theorem, that we will state later, arises from the symmetry of the studied system. Those symmetries are represented by a Lie group $G$ acting smoothly on the symplectic manifold M via $\psi : G \rightarrow \operatorname{Diff}(M)$. In addition, we assume that $G$ acts by symplectomorphism on $M$, i.e preserves the symplectic form.

\begin{defi}[Hamiltonian action]
\label{App:Ham_action}
    The action $\psi$ is said to be a Hamiltonian action if there exists a map $\mu : M \rightarrow \mathfrak{g}^*$ such that 
    \begin{itemize}
        \item For every $X \in \mathfrak{g}$, by denoting $\hat{\mu}(X) : p \mapsto  (\mu(p) \vert X)$, the vector field  $\tilde{X}$ on $M$ generated by the infinitesimal action satisfies   
        \begin{equation*}
        \nabla^\omega\hat{\mu}(X)= \tilde{X}
        \end{equation*}
        i.e $\tilde{X}$ is the symplectic gradient of $\hat{\mu}(X)$.
        
        \item $\mu$ is G-equivariant :
        
        $$ \mu \circ \psi_g = Ad_g^* \circ \mu \, \text{ for all } g \in G $$
        \end{itemize}
\end{defi}
Noether's theorem states that symmetries give rise to conserved quantites.

\begin{theo}[Noether's theorem]
    If the Hamiltonian $H$ is $G-$invariant, then the momentum maps $ \hat{\mu}(X) : p \in M \mapsto (\mu(p) \vert X) \in \R$ are conserved quantities along the Hamiltonian flow.
    
\end{theo}

We denote the orbit of G through $p\in M$ by $\operatorname{Orbit}_G(p)=\{ \psi_g(p) \mid g \in G\}$ and the stabilizer of $p \in M$ by the subgroup $G_p := \{ g \in G \mid \psi_g(p) = p \}$.
Being in the same orbit defines an equivalence relation $p \sim q $ if and only if $p$ and $q$ are on the same orbit, which defined an orbit space $M/G$. Then, we define the mapping $$\pi : \app{M}{M/G}{p}{\operatorname{Orbit}_G(p)}$$ called the point-orbit projection.
We can equip $M/G$ with the quotient topology, which is the weakest topology for which $\pi$ is continuous, i.e $\mathcal{U} \subset M/G$ is open if and only if $\pi^{-1}(\mathcal{U})$ is open in M. 

\begin{theo}[Marsden-Weinstein-Meyer \cite{MARSDEN1974121,marsden1994introduction}]
Let $(M,\omega)$ be a symplectic manifold and let $G$ be a Lie group with Hamiltonian action on $(M,\omega)$ with associated moment map $\mu : M \rightarrow \mathfrak{g}^*$. Let  $i : \mu^{-1}(0) \hookrightarrow M$ be the inclusion map. Assume that G acts freely and properly on $\mu^{-1}(0)$. Then,

\begin{itemize}
    \item the orbit space $\mu^{-1}(0)/G$ (sometimes denoted $M//G)$ is a symplectic manifold.
    \item $\pi : \mu^{-1}(0) \rightarrow \mu^{-1}(0)/G$ is a principal G-bundle
    \item there is a symplectic form $\omega_{red}$ on $\mu^{-1}(0)/G$ satisfying $i^* \omega = \pi^*\omega_{red}$.
\end{itemize}    
\end{theo}
The pair $(\mu^{-1}(0)/G,\omega_{red})$ is called the reduced space (or symplectic quotient).

\section{Variational problems in Banach space and Pontryagin principle}
\label{App:opt_cont}
Let $\mathcal{Q}$ be a Banach manifold, modeled on a Banach space $\mathbb{B}$, and $V$ a Banach space. Let $q_0\in\mathcal{Q}$. In this section, we are interested in the control system
\begin{equation}
    \label{dyn_Q}
    \dot{q}(t)=f(q(t),u(t)), \ q(0)=q_0
\end{equation}
where $u\in L^1(I,V)$, and $f: \mathcal{Q}\times V\to T\mathcal{Q}$ is $C^2$, and is a fiber bundle morphism over the identity of $\mathcal{Q}$ (this means that for all $(q,v)\in\mathcal{Q}\times V,$ $f(q,v)$ is in $T_q\mathcal{Q}$). We denote by $\mathcal{U}_{q_0}$ the set of admissible controls $u\in L^1(I,V)$ such that it leads to a maximal trajectory $q:[0,T]\to\R$ starting in $q(0)=q_0$ and such that $T>1$, and we suppose that the set $\mathcal{U}_{q_0}$ is open in $\mathcal{Q}$ (see for instance \cite{Arguillere2020,Younes2019, amann1990ode}). Let $\mathcal{D}:\mathcal{Q}\to \R^+$ be $C^1$. We are interested in this section on functionals $J:\mathcal{U}_{q_0}\to \R$ that take the form :
$$
J(u) = \int_I L(q(t),u(t))dt + \mathcal{D}(q(1))
$$
 such that $q\in AC_{L^1}([0,1],\mathcal{Q})$ satisfies the dynamic \eqref{dyn_Q}. We suppose that, for any charts $U\subset \mathcal{Q}$, there exists continuous functions $\gamma_0:\R\to\R$ and $\gamma_1:\mathcal{Q}\to\R$ such that $\gamma_0(0)=0$ and
\begin{equation}
    \label{Hdiff}
    \lVert dL_{q,u}-dL_{q',u'} \rVert \leq \gamma_0(\lvert q-q' \rvert_{\mathbb{B}})  +\gamma_1(q-q')\lvert u -u' \rvert_V
\end{equation}
for any $q,q'\in U$ and $u,u'\in V$. This condition \eqref{Hdiff} ensures that the mapping $J_0 : AC_{L^1}(I,\mathcal{Q})\times L^1(I,V)\to \R$ defined by $J_0(q,u) = \int_I L(q(t),u(t))dt$ is well defined and continuously differentiable \cite{arguillère_trélat_2017} and that 
\begin{equation}
    \label{derivative_J}
    dJ_0(q,u)(\delta q,\delta u) = \int_I\frac{\partial}{\partial q}L(q(t),u(t))\delta q(t) + \frac{\partial}{\partial u}L(q(t),u(t))\delta u(t) dt
\end{equation}
In this section, we want to minimize the functional $J$ and characterize its critical points, which is a problem of optimal control theory in Banach spaces. We start by proving a regularity property on the evolution map.

\begin{lemma}[Evolution map associated to the dynamic]
    The evolution map 
    $$
    \operatorname{Evol}_{\mathcal{Q}} : u\in \mathcal{U}_{q_0} \mapsto q^u \in AC_{L^1}(I,\mathcal{Q})
    $$
    where $q^u$ is solution of equation \eqref{dyn_Q}, is $C^1$. Furthermore, for $u\in\mathcal{U}_{q_0}, \delta u \in  L^1(I,V)$, its derivative $\delta q = T_u \operatorname{Evol}_{\mathcal{Q}}\cdot \delta u$ is solution of the linear Cauchy problem
\begin{equation}
\label{linq2}
    \delta \dot{q}(t) = \partial_q \left(f(q,u(t)\right)_{|q=q^u(t)}\delta q(t) + \partial_u(f(q^u(t),u)_{|u=u(t)}\delta u(t), \ \delta q(0)=0
\end{equation}
\end{lemma}

\begin{proof}

The proof is given in a less general case in \cite{gga} and is adapted here.
Let $u^*\in\mathcal{U}_{q_0}$, and $q^u=\operatorname{Evol}_{\mathcal{Q}}(u)$. Similarly to proposition \cite[prop. 2.6]{gga}, we consider a chart $\mathcal{V}\subset AC_{L^1}(I,\mathcal{Q})$ around $q^{u*}$, and we even actually reduce to the case $\mathcal{V}=AC_{L^1}(I,\mathbb{B})$. We consider the mapping
$$
C\left\{
\begin{array}{ccl}
     AC_{L^1}(I,\mathbb{B})\times\mathcal{U}_{q_0}&\longrightarrow& \mathbb{B} \times L^1(I,\mathbb{B})\\
     (q,u ) & \longmapsto & (q(0), \dot{q}-f(q,u))
\end{array}
\right.
$$
The mapping $C$ is $C^1$ since $f$ is $C^2$ and by \cite[prop 2.3]{glo}. Moreover, the derivative with regards to the first variable is given by
$$
\partial_q C(q,u)\delta q = (\delta q(0), \delta\dot{q}- \partial_qf(q,u)\delta q)
$$
and is a Banach isomorphism (by Banach-Schauder theorem). By implicit function theorem, and since $C(\operatorname{Evol}(u),u)=(q_0,0)$, there exists an open neighborhood $W\subset L^1(I,\mathbb{B})$ of $u^*$ such $\operatorname{Evol}$ is $C^l$, and the curve $\delta q = T_u\operatorname{Evol}\delta u$ is solution of the linear cauchy problem \eqref{linq2}.
\end{proof}

We study now the minimization of the functional $J$. Let recall that the Banach manifold $T^*\mathcal{Q}$ can be equipped with a canonical (weak) symplectic form $\omega$ (that we define introducing the Liouville form and its exterior derivative \cite{ARGUILLERE2015139}). In local charts of $T^*\mathcal{Q}$, the closed form $\omega$ is defined by 
$$
\omega_{q,p}\left((\delta q, \delta p), (\delta q',\delta p')\right) = (\delta p\,\vert\,\delta q') - (\delta p'\,\vert\,\delta q)
$$

Let $\mathcal{H}:T^*\mathcal{Q}\times V \to \R $ be the pre-Hamiltonian associated to the functional $J$ defined by
$$
\mathcal{H}(q,p,u) = \left(p\,\vert\, f(q,u)\right) - L(q,u)
$$
Since $f$ is $C^2$, and the duality pairing $((q,p),(q,X)) \in T^*G \oplus_Q TG \mapsto \left(p\,|\,X\right) \in \mathbb{R}$ is smooth, the Hamiltonian  $\mathcal{H}_{\mathcal{Q}} : T^*\mathcal{Q} \times V \longrightarrow \R$ is thus $C^2$. In local coordinates, the partial derivative $\partial_p \mathcal{H}_\mathcal{Q}(q,p,u)$  is given by:
$$
\forall \delta p \in T^*_q\mathcal{Q}, \ \partial_p \mathcal{H}_{\mathcal{Q}}(q,p,u) \delta p= (\delta p\,|\,f(q,u))
$$
so that $\partial_p \mathcal{H}_{\mathcal{Q}}(q,p,u)\simeq f(q,u) \in T_q\mathcal{Q}$. Therefore there exists a symplectic gradient $\nabla^{\omega}\mathcal{H}_{\mathcal{Q}}(q,p,u)\in T_{q,p}T^*\mathcal{Q}$ for every $u\in V$ . In canonical charts of $T^*\mathcal{Q}$, we have :
$$
    \nabla^{\omega}\mathcal{H}_{\mathcal{Q}}(q,p,u) = (\partial_p \mathcal{H}_{\mathcal{Q}}(q,p,u), \ -\partial_q \mathcal{H}_{\mathcal{Q}}(q,p,u))
$$
We get the following result :
\begin{theo}[Critical points of $J$] 
\label{app:critic_points}
Let $u^*\in \mathcal{U}_{q_0}$ be a critical point of $J$, that is $dJ(u^*)=0$, and denote by $q=\operatorname{Evol}(u)\in AC_{L^1}(I,Q)$ the solution of equation \eqref{dyn_Q}. Then there exists a covector $t\mapsto p(t)\in T_{q(t)}\mathcal{Q}^*$ in $AC_{L^1}(I,T\mathcal{Q}^*)$ such that $(q,p,u)$ satisfies the Hamiltonian equations :
\begin{equation}
    \label{Ham_eq_gen}
    \left\{ \begin{array}{ll}
       \left(\dot{q}, \ \dot{p}\right)   =  \nabla^{\omega}\mathcal{H}_{\mathcal{Q}}(q,p,u) \\
        \partial_u\mathcal{H}_{\mathcal{Q}}(q,p,u)  =  0 
    \end{array}
    \right.
\end{equation}
Conversely, if there exists such a covector $p$ such that $(q,p,u^*)$ satisfies the Hamiltonian equations \eqref{Ham_eq_gen}, then $u^*$ is a critical point of $J$.
\end{theo}
\begin{proof}
The proof is similar to the one given in \cite{Arguillere2020,arguillère_trélat_2017}. Let $u\in\mathcal{U}_{q_0}$, and $q=\operatorname{Evol}(u)$.
We introduce the covector $p_1 = -d\mathcal{D}(q_1)\in T^*_{q_1}\mathcal{Q}$ and we define $p\in AC_{L^1}(I,T^*\mathcal{Q})$ as the solution of the following linear Cauchy problem :
\begin{equation}
\label{CLdual_gen}
    \left\{
    \begin{array}{ll}
    \dot{p}(t)=-\partial_q \mathcal{H}_{\mathcal{Q}}(q(t),p(t),u(t))=-\big(\partial_q(f(q,u(t)))_{q=q(t)}\big)^*p(t) + \partial_q(L(q,u(t)))_{q=q(t)} \\
    p(1)=p_1
    \end{array}
    \right.
\end{equation}
    We now compute the differential of the mapping 
    $J$
    and we prove that for all $\delta u \in L^2(I,V)$,
    $$
    dJ(u)\delta u = -\int_I \partial_u\mathcal{H}(q(t),p(t),u(t))\delta u dt
    $$
    Let $\delta u \in L^2(I,V)$, and recall that $J(u)= J_0(\operatorname{Evol}(u),u)+ \mathcal{D}(\operatorname{Evol}(u)(1))$ so that we have 
    \begin{align*}
        dJ(u)\delta u &= dJ_0(q,u)(\delta q,\delta u) + d\mathcal{D}(q_1)(\delta q) \\
        &= \int_I \frac{\partial}{\partial q}L(q(t),u(t))\delta q(t) + \frac{\partial}{\partial u}L(q(t),u(t))\delta u(t) dt - (p_1\,\vert\, \delta q(1))
    \end{align*}
where $\delta q = T_u\operatorname{Evol}(\delta u)$ and thus satisfies the linear Cauchy problem \eqref{linq2}
$$
\delta q(0) = 0, \ \delta \dot{q}(t) = \partial_q \left(f(q,u(t)\right)_{|q=q(t)}\delta q(t) + \partial_u(f(q(t),u)_{|u=u(t)}\delta u(t).
$$
Now, since $t\mapsto p(t)$ is also solution of the linear Cauchy equation \eqref{CLdual_gen} and
    using integration by part, we find that
    \begin{align*}
    (p(1)&\,|\,\delta q(1)) = \left(p(0)\,|\,\delta q(0)\right) + \int_I \left(\dot{p}(t)\,|\,\delta q(t)\right) + \left(p(t)\,|\,\delta\dot{q}(t)\right)dt\\
    &=\int_I-\bigg((\partial_q\left(f(q,u(t)\right)_{|q=q(t)})^*p(t) -\partial_q(L(q,u(t)))_{q=q(t)}\,|\,\delta q(t)\bigg) \\ &+\bigg(p(t)\,|\,\partial_q \left(f(q,u(t)\right)_{|q=q(t)}\delta q(t) + \partial_u(f(q(t),u)_{|u=u(t)}\delta u(t) \bigg) dt \\
    &=\int_I \left(p(t)\,|\,\partial_u(f(q(t),u)_{|u=u(t)}\delta u(t)\right) + \partial_q (L(q,u(t))_{|q=q(t)}\delta q(t) dt 
    \end{align*}
Combining this with the derivative $dJ_0(q,u)$, we finally get
\begin{align*}
    dJ(u)\delta u &= \int_I \partial_u\left(L(q(t),u(t))\right)\delta u(t) - \left(p(t)\,|\,\partial_u\left(f(q(t),u(t))\right)\delta u(t)\right) dt \\
    &= -\int_I \partial_u\mathcal{H}(q(t),p(t),u(t))\delta u dt
\end{align*}
Therefore we have the following equivalence :
\begin{equation}
    dJ(u) = 0 \iff \forall t,\  \partial_u \mathcal{H}(q(t),p(t),u(t))=0
\end{equation}
which concludes the proof.
\end{proof}

\end{document}